\documentclass{amsart}
\usepackage{mathrsfs}
\usepackage{graphicx}
\usepackage{bbm}
\usepackage{relsize}
\usepackage{enumerate}
\usepackage{amssymb,amsmath,amsthm}
\usepackage{bm}
\usepackage{titlesec}
\usepackage{verbatim}
\usepackage{mathrsfs}
\usepackage[utf8]{inputenc}
\usepackage[english]{babel}
\usepackage{color,soul}
\usepackage{chngcntr} 
\usepackage[letterpaper, portrait, margin=1in]{geometry}

\newcommand{\ZZ}{\mathbb{Z}}
\newcommand{\RR}{\mathbb{R}}

\newcommand{\MM}{\mathbb{M}}

\newcommand{\TTT}{\mathcal{T}}

\newcommand{\MMM}{\mathcal{M}}

\newtheorem{theorem}{Theorem}
\newtheorem*{theorem*}{Theorem}

\newtheorem{prop}[theorem]{Proposition}

\newtheorem{cor}[theorem]{Corollary}
\newtheorem{lemma}[theorem]{Lemma}

\newtheorem{defn}{Definition}

\newtheoremstyle{named}{}{}{\upshape}{}{\bfseries}{}{.5em}{#1 \thmnote{(#3)}}
\theoremstyle{named}

\numberwithin{theorem}{section}
\numberwithin{equation}{section}

\usepackage[pdftex,plainpages=false,hypertexnames=false,pdfpagelabels]{hyperref}
\newcommand{\arxiv}[1]{\href{http://arxiv.org/abs/#1}{\tt arXiv:\nolinkurl{#1}}}
\newcommand{\arXiv}[1]{\href{http://arxiv.org/abs/#1}{\tt arXiv:\nolinkurl{#1}}}

\newcommand{\googlebooks}[1]{(preview at \href{http://books.google.com/books?id=#1}{google books})}

\DeclareMathOperator{\conv}{conv}

\DeclareMathOperator{\rank}{rank}

\titleformat{\section}
{\normalfont\Large\bfseries}{\thesection}{1em}{}
\titleformat{\subsection}
{\normalfont\large\bfseries}{\thesubsection}{1em}{}
\titleformat{\subsubsection}
{\normalfont\normalsize\bfseries}{\thesubsubsection}{1em}{}

\title{The moduli space of tropical curves with fixed Newton polygon}
\author{Desmond Coles, Neelav Dutta, Sifan Jiang, Ralph Morrison and Andrew Scharf}

\usepackage[utf8]{inputenc}
\begin{document}

\begin{abstract}  Given a lattice polygon, we study the moduli space of all tropical plane curves with that Newton polygon.  We determine a formula for the dimension of this space in terms of  combinatorial properties of that polygon.  We prove that if this polygon is nonhyperelliptic or maximal and hyperelliptic, then this formula matches the dimension of the moduli space of nondegenerate algebraic curves with the given Newton polygon.
\end{abstract}

\maketitle

\section{Introduction}  Tropical geometry is a powerful combinatorial tool for studying algebraic geometry.  It associates to a classical variety a ``skeletonized'' version of that variety, whose combinatorial properties reflect algebro-geometric ones.  In the case of studying plane curves (or more generally curves on toric surfaces), the tropical object is called a tropical plane curve.  It is a subset of $\mathbb{R}^2$ that has the structure of a weighted, balanced polyhedral complex of dimension $1$.  Any tropical plane curve $C$ is defined by a tropical polynomial in two variables over the min-plus semiring, and is dual to a subdivision of the Newton polygon $\Delta$ of that polynomial. The tropical curve $C$ contains a distinguished metric graph $G$, called its skeleton, which is the smallest subset of $C$ that admits a deformation retract. If the subdivision of $\Delta$ is a unimodular triangulation, we call $C$ smooth.  In the event that $C$ is smooth, then $C$, as well as $G$, has genus (that is, first Betti number) equal to the number of interior lattice points of $\Delta$. This is illustrated in Figure \ref{figure:first_honeycomb}, which shows a regular unimodular triangulation of a polygon with $9$ interior lattice points on the left, a dual smooth tropical plane curve of genus $9$ in the middle, and the curve's skeleton on the right.  We remark that for a planar graph, the genus can also be characterized as the number of bounded faces in any planar embedding of that graph.

	\begin{figure}[hbt]
   		 \centering
        \includegraphics[scale=1]{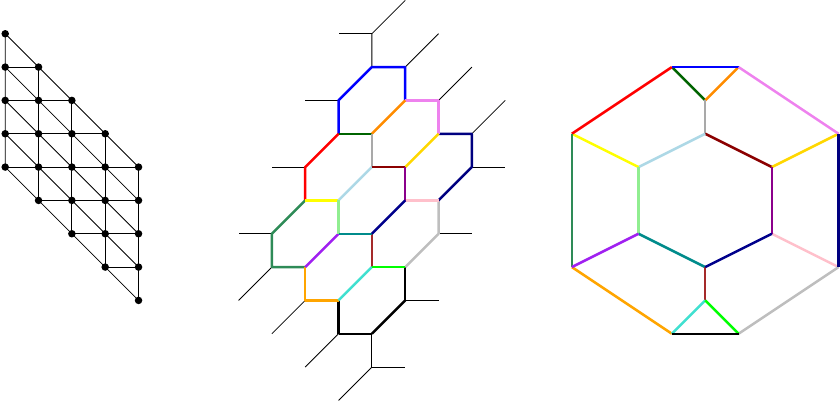}
		\caption{A triangulated lattice polygon, a dual tropical curve, and its skeleton}
		\label{figure:first_honeycomb}
	\end{figure}

The authors of \cite{BJMS} introduced the \emph{moduli space of tropical plane curves of genus $g$}, denoted $\mathbb{M}_g^{\textrm{planar}}$, for any integer $g\geq 2$. As a set, this consists of all metric graphs of genus $g$ that appear as the skeleton of a smooth tropical plane curve, up to closure.  Geometrically, it is the locus of such graphs inside of $\mathbb{M}_g$, the moduli space of metric graphs of genus $g$:
	\[\mathbb{M}_g^{\textrm{planar}}\subset\mathbb{M}_g.\]
The space $\mathbb{M}_g^{\textrm{planar}}$ can be written as a finite union of simpler polyhedral spaces.  In particular, we can write
	\[\mathbb{M}_g^{\textrm{planar}}=\bigcup_{\Delta}\mathbb{M}_\Delta,\]
where the union is taken over all lattice polygons $\Delta$ with exactly $g$ interior lattice points, and where $\mathbb{M}_\Delta$ is the closure of the set of all metric graphs that are the skeleton of a smooth tropical curve with Newton polygon $\Delta$.  By considering lattice polygons up to equivalence, we may take this union to be finite, as discussed in \cite[Proposition 2.3]{BJMS}. We may also restrict our consideration to \emph{maximal} polygons, which are those that are not contained in any other lattice polygon with the same set of interior lattice points.  Sometimes one sorts the polygons $\Delta$ into two groups:  the \emph{hyperelliptic} polygons, which have all interior lattice points collinear; and the \emph{nonhyperelliptic} polygons, which are not hyperelliptic.

Each $\mathbb{M}_\Delta$ admits its own decomposition:
	\[\mathbb{M}_\Delta=\bigcup_{\mathcal{T}}\mathbb{M}_\mathcal{T},\]
where the union is taken over all regular unimodular triangulations $\mathcal{T}$ of $\Delta$, and where $\mathbb{M}_\mathcal{T}$ is the closure of the set of all metric graphs that are the skeleton of a smooth tropical curve dual to $\mathcal{T}$.  Taken together, we have
	\[\mathbb{M}_g^{\textrm{planar}}=\bigcup_\Delta\bigcup_\mathcal{T}\mathbb{M}_\mathcal{T}.\]
The dimension of $\mathbb{M}_g^{\textrm{planar}}$ is then the maximum of $\dim(\mathbb{M}_\mathcal{T})$, taken over all regular unimodular triangulations $\mathcal{T}$ of all polygons $\Delta$ with $g$ interior lattice points.  We remark that our notation differs from that in \cite{BJMS}:  they use $P$ for a polygon and $\Delta$ for a triangulation of the polygon.  We choose our notation to more closely mirror \cite{nondegeneracy}, who use $\Delta$ for a polygon.

Since each trivalent graph of genus $g$ has $3g-3$ edges, $\mathbb{M}_g$ is a $(3g-3)$-dimensional polyhedral space. By \cite{BJMS}, the dimension of $\mathbb{M}_g^{\textrm{planar}}$ is $d(g)$, where
\[d(g)=\begin{cases}
3&\textrm{ if $g=2$}\\
6&\textrm{ if $g=3$}\\
16&\textrm{ if $g=7$}\\
2g+1&\textrm{ otherwise}.
\end{cases}\]
Their proof of this proceeds as follows.  First, to show that $\dim\left(\mathbb{M}_g^{\textrm{planar}}\right)$ is bounded above by the claimed numbers, they note that $\mathbb{M}_g^{\textrm{planar}}$ is contained in the tropicalization of the moduli space of nondegenerate curves $\mathcal{M}_g^{\textrm{nd}}$ introduced in \cite{nondegeneracy}.  It is known that $\text{dim}\left(\mathcal{M}_g^{\textrm{nd}}\right)=d(g)$ for all $g$ by \cite[Theorem 12.2]{nondegeneracy}, and since dimension is preserved under tropicalization, we have
\[\dim\left(\mathbb{M}_g^{\textrm{planar}}\right)\leq \dim\left(\mathcal{M}_g^{\textrm{nd}}\right)=d(g).\]
It remains to show that $\dim\left(\mathbb{M}_g^{\textrm{planar}}\right)$ is at least as large as $d(g)$. For each $g\geq 2$, \cite{BJMS} construct a polygon $\Delta$ with $g$ interior lattice points together with a regular unimodular triangulation $\mathcal{T}$ of $\Delta$ such that $\dim\left(\mathbb{M}_\mathcal{T}\right)=d(g)$.  It follows that 
\[\dim\left(\mathbb{M}_g^{\textrm{planar}}\right)\geq d(g),\] implying equality.

The polygons used by \cite{BJMS} to achieve this lower bound are called \emph{honeycomb polygons}.  These are polygons admitting a triangulation whose primitive triangles are all translations of the triangles with vertices $(0,0),(1,0),(0,1)$ and $(0,0),(-1,0),(0,-1)$; such a triangulation appears in Figure \ref{figure:first_honeycomb}.  It turns out that for a honeycomb triangulation $\mathcal{T}$ of a honeycomb polygon $\Delta$, the dimension of $\mathbb{M}_\mathcal{T}$ can be expressed in terms of data from $\Delta^{(1)}$, the convex hull of the interior lattice points of $\Delta$.  We call $\Delta^{(1)}$ the \emph{interior polygon} of $\Delta$.

\begin{prop}[\cite{BJMS}, Lemma 4.2]\label{Thm:HoneycombDim}  Let $\Delta$ be a honeycomb polygon, and $\mathcal{T}$ the honeycomb triangulation.  Suppose $\Delta^{(1)}$ has $g$ lattice points, $g^{(1)}$ interior lattice points, and $b$ boundary lattice points that are not vertices.  Then
\[\dim\left(\mathbb{M}_{\mathcal{T}}\right)=3g-3-2g^{(1)}-b.\]
\label{theorem:bjms}
\end{prop}
Since $3g-3$ is the dimension of $\mathbb{M}_g$, this result says that each interior point of $\Delta^{(1)}$ contributes $2$ to the codimension of $\mathbb{M}_{\mathcal{T}}$, while each non-vertex boundary point of $\Delta^{(1)}$ contributes $1$ to the codimension.  For example, the lattice polygon in Figure \ref{figure:first_honeycomb} has $g=9$, $g^{(1)}=1$, and $b=4$.  Thus $\mathbb{M}_\mathcal{T}$ has dimension $\dim(\mathbb{M}_\mathcal{T})=3\cdot 9-3-2\cdot 1-4=18$, and sits inside the $24$-dimensional space $\mathbb{M}_9$.  Since $18=\dim(\mathbb{M}_\mathcal{T})\leq\dim(\mathbb{M}_\Delta)\leq\dim(\mathbb{M}_9^{\textrm{planar}})=19$, we can deduce that $\dim(\mathbb{M}_\Delta)$ is either $18$ or $19$.

Our first main result provides a simple way to compute $\dim(\mathbb{M}_\mathcal{T})$ for any regular unimodular triangulation $\mathcal{T}$.  Certain edges, defined below, play a key role.

\begin{defn}An edge $e$ in $\mathcal{T}$ is called \emph{radial} if it connects an interior lattice point of $\Delta$ to a boundary lattice point of $\Delta$, such that $e\cap \Delta^{(1)}$ consists of a single point.
\end{defn}

\begin{theorem}\label{theorem:lower_bound_mt}
Let $\mathcal{T}$ be a regular unimodular triangulation of a nonhyperelliptic lattice polygon $\Delta$.  Let $b_1$ be the number of lattice points in $\partial{\Delta}^{(1)}$ incident to only one radial edge in $\mathcal{T}$, and let $b_2$ be the number of lattice points in $\partial{\Delta}^{(1)}$ incident to two or more radial edges, all of whose endpoints are mutually collinear.  Then
\[\dim\left(\mathbb{M}_{\mathcal{T}}\right)=3g-3-2g^{(1)}-2b_1-b_2.\]\end{theorem}
In the special case that $\mathcal{T}$ is a honeycomb triangulation, we have $b_1=0$ and $b_2=b$, thus recovering Proposition \ref{theorem:bjms}.

Since $\dim\left(\mathbb{M}_{\Delta}\right)$ is the maximum of $\dim\left(\mathbb{M}_{\mathcal{T}}\right)$ over all regular unimodular triangulations $\TTT$ of $\Delta$, we wish to find a regular triangulation of $\Delta$ minimizing the value of $2b_1+b_2$.  We construct and analyze such an optimal triangulation,
leading us to the following theorem for maximal nonhyperelliptic polygons. It is framed in terms of the number of \emph{column vectors} of a polygon, which are translation vectors that keep a polygon contained within itself after deleting a face; see Section \ref{section:discrete_geometry} for a more precise definition.  

\begin{theorem}\label{theorem:lower_bound_mdelta}  	Let $\Delta$ be a maximal nonhyperelliptic polygon with $g$ interior lattice points, $r$ boundary lattice points, and $c(\Delta)$ column vectors. Then we have
    \[ \dim(\mathbb{M}_{\Delta}) = g - 3 - c(\Delta)+r. \]
\end{theorem}

We find a similar, though more complicated, formula for $\dim(\MM_\Delta)$ when $\Delta$ is a nonmaximal nonhyperelliptic polygon.  These formulas help us relate these tropical moduli spaces to algebraic ones.  As defined in \cite{nondegeneracy}, $\mathcal{M}^{\textrm{nd}}_{g}$ is constructed as a union of spaces $\MMM_{\Delta}$ which are the moduli spaces of nondegenerate curves with fixed Newton polygon $\Delta$.  It was noted in \cite{BJMS} that $\dim(\mathbb{M}_\Delta)\leq \dim(\mathcal{M}_\Delta)$, with equality known only for particular families of honeycomb polygons, such as rectangles and isosceles right triangles \cite[\S 4]{BJMS}.  They posed as an open question whether or not these dimensions are always equal \cite[Question 8.6(1)]{BJMS}.  Our main theorem answers this question in the affirmative for most lattice polygons.

\begin{theorem}\label{Thm:MainTheorem}  Let $\Delta$ be a nonhyperelliptic lattice polygon of genus $g\geq 2$. Then
\[\dim\left(\mathcal{M}_{\Delta}\right)=\dim\left(\mathbb{M}_{\Delta}\right).\]
The same holds if $\Delta$ is maximal and hyperelliptic.
\end{theorem}

This can be thought of as a moduli-theoretic analog of Mikhalkin's work on tropical curve counting \cite{mikhalkin}, which showed that the tropical and algebraic counts for the number of plane curves with prescribed genus and Newton polygon passing through a generic collection of points agree with one another. Indeed, it is possible that the results and techniques of that paper could be used to give an alternate proof of Theorem \ref{theorem:lower_bound_mdelta}.  The approach we take in this paper has the added benefit of giving us Theorem \ref{theorem:lower_bound_mt} as an intermediate result, thus allowing us to compute $\dim\left(\mathbb{M}_\mathcal{T}\right)$ based on purely combinatorial properties of the triangulation~$\mathcal{T}$.


Our paper is organized as follows.  In Section \ref{section:discrete_geometry} we present the necessary background on polygons, triangulations, tropical curves, and algebraic moduli spaces.  In Section \ref{section:computing_dimension} we provide a method for computing the dimension of $\mathbb{M}_\mathcal{T}$ for any regular unimodular triangulation $\mathcal{T}$ of a nonhyperelliptic polygon.  In Section \ref{section:beehive} (respectively Section \ref{section:nonmaximal}) we construct regular triangulations of maximal (respectively nonmaximal) nonhyperelliptic polygons achieving the maximum possible dimension of $\mathbb{M}_\mathcal{T}$ in order to compute $\dim(\mathbb{M}_\Delta)$, and we show that this matches $\dim(\mathcal{M}_\Delta)$.   Finally, in Section \ref{section:hyperelliptic} we prove that our main theorem also holds for maximal hyperelliptic polygons.

\medskip

\noindent \textbf{Acknowledgements}  The authors are grateful for their support from the 2017 SMALL REU at Williams
College, and from the NSF via NSF Grant DMS1659037.  The authors also thank Wouter Castryck and John Voigt  for helpful discussions on moduli spaces of algebraic curves, as well as the referees and editors for many helpful comments and suggestions.

\section{Discrete and Algebraic Geometry}
\label{section:discrete_geometry}

In this section we present background and terminology coming from discrete and tropical geometry, as well as from algebraic geometry.  First we recall some results on lattice polygons, and then on subdivisions and triangulations thereof, to which tropical curves are dual.  Then we briefly recall the algebro-geometric topics pertinent to our results.

\subsection{Lattice polygons, subdvisions, and tropical curves}

A convex polygon $\Delta\subset \mathbb{R}^2$ is the convex hull of finitely many points.  If all the vertices of a polygon have integer coordinates, we refer to it as a \emph{lattice polygon}. Throughout this paper, all polygons will be assumed to be two-dimensional convex lattice polygons, unless otherwise stated. If a polygon has $g$ interior lattice points, we say that the polygon has \emph{genus $g$}. The \emph{interior polygon} $\Delta^{(1)}$ of a lattice polygon $\Delta$ is the convex hull of the $g$ lattice points in the interior of $\Delta$.  Depending on the number and arrangement of these lattice points, $\Delta^{(1)}$ is either empty, a single point, a line segment, or a two-dimensional lattice polygon.  Following the terminology of \cite{pushingout}, if $\dim(\Delta^{(1)})=2$ then we call $\Delta$ \emph{nonhyperelliptic},  and if $\dim(\Delta^{(1)})=1$ we call $\Delta$ \emph{hyperelliptic}.   We say $\Delta$ is \emph{maximal} if it is not properly contained in another lattice polygon with the same interior polygon.


We can also describe a lattice polygon $\Delta$ as a finite intersection of half-planes.  If $\tau \subset \Delta$ is a one-dimensional face, then $\tau$ corresponds to a half-plane $\mathcal{H}(\tau)$ in $\mathbb{R}^{2}$, namely \[\mathcal{H}(\tau) = \{(x,y) \in \mathbb{R}^{2}| \alpha_{\tau}x + \beta_{\tau}y \leq c_{\tau}\},\]
so that $\Delta = \bigcap_{\tau \in \Delta} \mathcal{H}(\tau)$. For each $\tau$ we may obtain a unique collection of integers $\alpha_{\tau}, \beta_{\tau}, c_{\tau}$ by stipulating $\gcd(\alpha_{\tau}, \beta_{\tau}) = 1$. We define the \emph{relaxed polygon} of $\Delta$ as \[\Delta^{(-1)} = \bigcap_{\tau \in \Delta} \mathcal{H}^{(-1)}_\tau, \]
where \[\mathcal{H}^{(-1)}_{\tau} = \{(x,y) \in \mathbb{R}^{2}| \alpha_{\tau}x + \beta_{\tau}y \leq c_{\tau} + 1\}.\]
The boundary of $\mathcal{H}^{(-1)}_{\tau}$ is $\tau^{(-1)}=\{(x,y) \in \mathbb{R}^{2}| \alpha_{\tau}x + \beta_{\tau}y = c_{\tau} + 1\}$. Given that $\tau$ was a 1-dimensional face of $\Delta$, in an abuse of notation we may use $\tau^{(-1)}$ to refer to a face of $\Delta^{(-1)}$.  It is worth remarking that if $\Delta$ is a lattice polygon, it is not necessarily the case that $\Delta^{(-1)}$ is a lattice polygon.  We also note that although every one-dimensional face of $\Delta^{(-1)}$ is of the form $\tau^{(-1)}$, not every $\tau^{(-1)}$ is a one-dimensional face of $\Delta^{(-1)}$.  See Figure \ref{figure:weird_polygons} for illustrations of these two phenomena.


\begin{figure}[hbt]\label{figure:weird_polygons}
\centering
\includegraphics[scale=1]{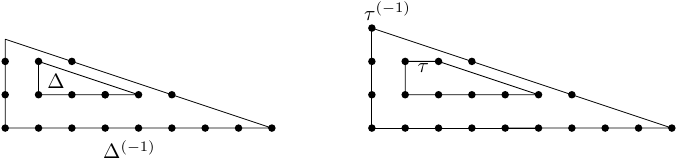}
\caption{A lattice polygon $\Delta$ whose relaxed polygon is not a lattice polygon; and a lattice polygon with a face $\tau$ such that $\tau^{(-1)}$ is a vertex of $\Delta^{(-1)}$}
\end{figure}


\begin{lemma}[\cite{Koelman}, \S 2.2 and \cite{hs}, Lemmas 9 and 10]  Let $\Delta$ be a nonhyperelliptic lattice polygon.  Then
$\Delta$ is maximal if and only if $\Delta$ is the relaxed polygon of $\Delta^{(1)}$; that is, if and only if $\Delta^{(1)(-1)} = \Delta$.
\end{lemma}
It follows that for any nonhyperelliptic polygon $\Delta$, there exists a unique maximal lattice polygon of the same genus containing it, namely $\Delta^{(0)}:=\Delta^{(1)(-1)}$.  Such a polygon $\Delta$ is illustrated on the left in Figure \ref{figure:nonmaximal_to_maximal}, followed by its interior polygon $\Delta^{(1)}$, followed by the relaxed polygon of the interior polygon $\Delta^{(1)(-1)}$.
\begin{figure}[hbt]
   		 \centering
		\includegraphics{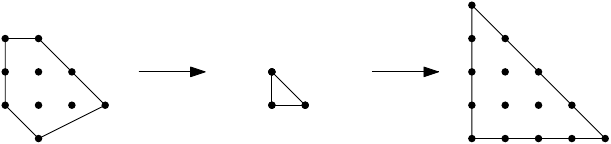}
	\caption{A nonmaximal polygon, its interior polygon, and the corresponding maximal polygon}
	\label{figure:nonmaximal_to_maximal}
\end{figure}

If $\Delta$ is a nonhyperelliptic lattice polygon whose interior polygon $\Delta^{(1)}$ has a one-dimensional face $\tau$,then  $\tau^{(-1)}$ does appear in $\Delta^{(-1)}$, at least with one lattice point.

\begin{lemma}\label{pushoutnonempty} The subset $\tau^{(-1)}$ in $\Delta$ is nonempty.

\end{lemma}

\begin{proof}  To avoid notational confusion, let $\tau^{(-1)}$ denote the boundary of $\mathcal{H}_\tau^{(-1)}$, and let $\tilde{\tau}_i^{(-1)}$ denote $\tau^{(-1)}\cap \Delta$.
Without loss of generality we can assume $\tau$ lies along the $x$-axis with $\Delta^{(1)}$ contained in the upper half plane, so that $\tau^{(-1)}$ is the line defined by $y=-1$. Any lattice polygon contained in $\Delta^{(0)}$ not intersecting $\tau^{(-1)}$ must be entirely contained in the upper half plane defined by $y\geq 0$, and thus could not contain $\Delta^{(1)}$ in its interior.  Thus $\tilde{\tau}^{(-1)}$ is nonempty.
\end{proof}


\begin{figure}[hbt]
   		 \centering
		\includegraphics[scale=.5]{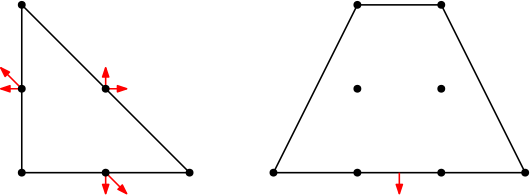}
	\caption{Column vectors associated to the faces of two polygons}
	\label{figure:column_vectors}
\end{figure}

Certain vectors, defined below. will play a key role in studying polygons and moduli spaces.

\begin{defn}
A nonzero vector $v\in \ZZ^{2}$ is a \emph{column vector} of $\Delta$ if there exists a facet $\tau \subset \Delta$ (referred to as the base facet) such that \[v + ((\Delta-\tau)\cap \ZZ^{2}) \subset \Delta.\]
\end{defn}

Two polygons are illustrated in Figure \ref{figure:column_vectors}, along with all their column vectors.  
For a face $\sigma$ of a lattice polygon, let $|\sigma|$ denote the number of lattice points in $\sigma$.  It turns out that the difference between $|\tau_i^{(-1)}|$ and $|\tau_i|$ encodes information about the column vectors associated to $\tau_i^{(-1)}$.

\begin{prop}\label{proposition:counting_columns}
Let $\Delta$ be a maximal nonhyperelliptic polygon, and $\tau_i$ a face of the interior polygon $\Delta^{(1)}$.  If $|\tau_i^{(-1)}|-1>|\tau_i|$, then the number of column vectors associated to $\tau_i^{(-1)}$ is equal to $|\tau_i^{(-1)}|-1-|\tau_i|$.  If  $|\tau_i^{(-1)}|-1\leq|\tau_i|$, then there are no column vectors associated to $\tau_i^{(-1)}.$
\end{prop}

This follows from the proof of \cite[Lemma 10.5]{nondegeneracy}.
As an example, the maximal polygon in Figure \ref{figure:nonmaximal_to_maximal} has $|\tau_i^{(-1)}|-1-|\tau_i|=5-1-2=2$ for all $i$, and indeed each facet has two column vectors: they are the same as for the  triangle in Figure \ref{figure:column_vectors}.

We now recall results and terminology on subdivisions of polygons. A \emph{subdivision} of a lattice polygon $\Delta$ is a partition of $\Delta$ into finitely many lattice subpolygons with the structure of a polyhedral complex, so that two polygons intersect at a shared face (either the empty set, a vertex, or an edge).  If all two-dimensional cells in a subdivision are triangles, that subdivision is called a \emph{triangulation}.  We refer to a triangulation $\mathcal{T}$ as \emph{unimodular} if all the triangles in $\mathcal{T}$ have area $\frac{1}{2}$. Since we are working in two dimensions, a triangulation is unimodular if and only if it cannot be further subdivided using cells whose vertices are lattice points \cite[Corollary 9.3.6]{triangulations}.  A key fact is that any non-unimodular subdivision of a polygon can be refined to a unimodular triangulation.

\begin{figure}[hbt]
   		 \centering
		\includegraphics[scale=1]{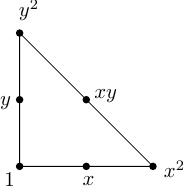}
        \quad\quad\quad
        \includegraphics[scale=.5]{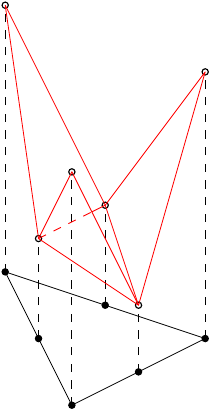}
        \quad\quad\quad
        \includegraphics[scale=1,trim=0 -.25cm 0 0]{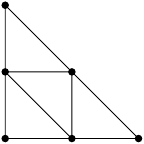}
	\caption{A lattice polygon with a height function and corresponding regular triangulation}\label{RegularSubdivision}
\end{figure}

We say a subdivision $\mathcal{S}$ is \emph{regular} if there exist a height function $\omega:\Delta\cap\mathbb{Z}^2\rightarrow\mathbb{R}$ such that the lower convex hull of its image, projected back onto $\Delta$, yields $\mathcal{S}$.  In this case we say  $\omega$ \emph{induces} $\mathcal{S}$.  This process is illustrated in Figure \ref{RegularSubdivision}.  We recall that given a regular subdivision of a polygon, there exists a unimodular refinement of that subdivision that is regular \cite[Proposition 2.3.16]{triangulations}.

Given a regular subdivision $\mathcal{T}$ of $\Delta$, the \emph{secondary cone} $\Sigma(\mathcal{T})$ of $\mathcal{T}$ is the collection of all height functions in $\mathbb{R}^{\Delta\cap\mathbb{Z}^2}$ that induce the subdivision $\mathcal{T}$.  The set $\Sigma(\mathcal{T})$ is indeed a cone, relatively open.  In the case that $\mathcal{T}$ is a unimodular triangulation, we can give a nice characterization of the inequalities defining $\Sigma(\mathcal{T})$.  If $P_1=(x_1,y_1),P_2=(x_2,y_2),P_3=(x_3,y_3),P_4=(x_4,y_4)$ are points in $\Delta\cap\mathbb{Z}^2$ such that the triangles formed by  $P_1,P_2,P_3$ and $P_2,P_3,P_4$ are unimodular triangles in $\mathcal{T}$ and if $\omega\in \Sigma(\mathcal{T})$, then we have that

	\begin{equation}\label{equation:determinant}\det
    	\left(
    	\begin{matrix}				 1 & 1 & 1 & 1 \\
				 x_1 & x_2 & x_3 & x_4 \\
 				 y_1 & y_2 & y_3 & y_4 \\
			  \omega (P_1) &  \omega (P_2) &  \omega (P_3) &  \omega (P_4)
 		\end{matrix}
        \right)
      >0.
    \end{equation}
The solution set to all such inequalities is exactly $\Sigma(\TTT)$.

Our interest in subdivisions arises from \emph{tropical plane curves}, defined over the min-plus semiring $(\overline{\mathbb{R}},\oplus,\odot)$, where $\overline{\mathbb{R}}=\mathbb{R}\cup\{\infty\}$, where $a\oplus b=\min\{a,b\}$, and where $a\odot b=a+b$. A tropical polynomial $f(x,y)$ in two variables $x$ and $y$ is a tropical sum $f(x,y)=\bigoplus_{i,j\in\mathbb{Z}} a_{ij}\odot x^i\odot y^j$, where $a_{ij}\in\overline{\mathbb{R}}$ with only finitely many $a_{ij}\neq\infty$; in classical notation, $f(x,y)=\min_{i,j\in\mathbb{Z}}(a_{ij}+ix+jy)$. The \emph{tropical curve} defined by $f$ is the set of all points $(x,y)$ in $\mathbb{R}^2$ where the minimum is achieved at least twice. This set can be endowed with the structure of a one-dimensional polyhedral complex.

Any tropical curve $C$ is dual to a regular subdivision of a lattice polygon,  in particular the subdivison of the Newton polygon of $f(x,y)$ induced by the height function $\omega$ given by the coefficients of $f(x,y)$ \cite[Proposition 3.1.6]{ms}.   If the dual subdivision of a tropical curve $C$ is a unimodular triangulation $\mathcal{T}$, we say that $C$ is \emph{smooth}.  Assume now that the genus $g$ of the Newton polygon $\Delta$ is at least $2$.  We may think of $C$ as a metric graph, where each edge has a length associated to it based on the $\mathbb{Z}^2$ lattice  (where rays have infinite length).  Let $G$ be the minimal metric graph onto which $C$ admits a deformation retract.  This graph $G$ can be constructed by first removing all rays, and then iteratively removing any $1$-valent vertices and their attached edges.  This yields a graph with $2$-valent and $3$-valent vertices.  Concatenate any edges joined at a $2$-valent vertex, adding their edge lengths.  Since $g\geq 2$ and since $C$ is smooth, we will end up with a metric graph that has $2g-2$ vertices, all $3$-valent.  This metric graph is called the \emph{skeleton} of $C$.  Note that for any tropical curves with the same dual triangulation, this skeletonization process will run exactly the same, except possibly for keeping track of different edge lengths.  An example of a regular unimodular triangulation, a dual smooth tropical curve $C$, and the metric skeleton $G$ are illustrated in Figure \ref{figure:genus2_complete_example}.  All bounded edges in the tropical curve have length $1$, while the metric skeleton has one edge of length $1$ and two of length $5$.  Note that some bounded edges in $C$ do not contribute to the edge lengths in~$G$.  

	\begin{figure}[hbt] 
    	\centering
		\includegraphics[scale=1]{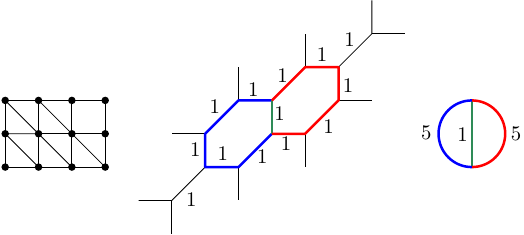}
        \caption{A unimodular triangulation, a dual tropical curve with bounded edge lengths labelled, and the associated metric skeleton}
        \label{figure:genus2_complete_example}
	\end{figure}
	
As it is a metric graph, the skeleton of $C$ is a point in  $\mathbb{M}_g$, the moduli space of all metric graphs of genus $g$.  Following \cite{BJMS}, we define the moduli space $\mathbb{M}_\mathcal{T}$ as the closure of the set of all points in  $\mathbb{M}_g$ that are skeletons of smooth tropical plane curves dual to $\mathcal{T}$.  For the triangulation $\mathcal{T}$ in Figure \ref{figure:genus2_complete_example}, we claim that $\mathbb{M}_\TTT$ consists of all metric graphs with that combinatorial type of graph:  any three lengths $a,b,c\in\mathbb{R}$ with $a,c>b$ can be achieved by extending or contracting edges in the form of the tropical curve, and up to closure this gives us all metrics on the combinatorial graph.

A more constructive characterization of $\mathbb{M}_\mathcal{T}$ appears in \cite[\S 2]{BJMS}, which we briefly recall here.  Let $\Delta$ be some lattice polygon, and $A=\Delta \cap \ZZ^2$. Let $\TTT$ be a regular subdivision of $\Delta$ induced by $\omega: A \rightarrow \RR$, and let $E$ denote the set of bounded edges in any tropical curve dual to $\Delta$ . There exist linear maps $\lambda:\mathbb{R}^A\rightarrow \mathbb{R}^E$ and $\kappa:\mathbb{R}^E\rightarrow \mathbb{R}^{3g-3}$ such that:  for any $p\in \Sigma(\mathcal{T})$, we have $\lambda(p)$ is the edge lengths in the tropical curve defined by $p$; and for any assignment $q$ to the lengths of the tropical curve, $\kappa(q)$ is the lengths on the corresponding skeleton.  Thus $\kappa\circ\lambda(\Sigma(\mathcal{T}))=\mathbb{M}_\mathcal{T}$, where we identify its image in $\mathbb{R}^{3g-3}$ with the corresponding subset of $\mathbb{M}_g$; see \cite[Proposition 2.2]{BJMS}.  From here, we define $\mathbb{M}_{\Delta}$
as  \[\mathbb{M}_{\Delta}=\bigcup_{\mathcal{T}}\mathbb{M}_\mathcal{T}\subseteq \mathbb{M}_g,\]
where the union is taken over all regular unimodular triangulations $\mathcal{T}$ of $\Delta$.  The moduli space  $\mathbb{M}_g^{\textrm{planar}}$ is then the union over all $\mathbb{M}_{\Delta}$ where $\Delta$ has genus $g$, a union that can be taken to be finite.

\subsection{Algebraic Geometry}
\label{subsection:ag}

We close this section by briefly recalling the algebro-geometric context for this work, as well as some useful results; see \cite[\S 3]{BJMS} for more details.  To each of $\mathbb{M}_g$,  $\mathbb{M}_g^{\textrm{planar}}$, and $\mathbb{M}_\Delta$, there is an analogous algebraic space:  the moduli space $\mathcal{M}_g$ of algebraic curves of genus $g$; the moduli space $\mathcal{M}_g^{\textrm{nd}}$ of non-degenerate plane curves of genus $g$ \cite{nondegeneracy}; and the moduli space $\mathcal{M}_\Delta$ of non-degenerate plane curves with Newton polygon $\Delta$, where $\Delta$ has genus $g$.   In each case, there is a tropicalization map from the algebraic space $\mathcal{M}$ to the tropical space~$\mathbb{M}$:\begin{equation*}
 \begin{matrix}
&&\mathcal{M}_\Delta  & \subseteq &  \mathcal{M}_g^{\rm nd} &  \subseteq & \mathcal{M}_g \\
&& \downarrow && \downarrow & & \downarrow\\
&&\text{trop}(\mathcal{M}_\Delta)  & \subseteq & 
\text{trop}(\mathcal{M}_g^{\rm nd}) &  \subseteq & 
\text{trop}(\mathcal{M}_g)  \smallskip \\
&& \text{\rotatebox{90}{$\subseteq$}} && \text{\rotatebox{90}{$\subseteq$}} & & \text{\rotatebox{90}{$=$}} \\
\mathbb{M}_{\mathcal{T}}&\subseteq&\mathbb{M}_\Delta  & \subseteq &  \mathbb{M}_g^{\rm planar} &  \subseteq & \mathbb{M}_g \\
\end{matrix}
\end{equation*}
In general, the containments between the second and third rows can be strict.  For example, suppose $g=3$ and $\Delta=\textrm{conv}\left((0,0),(4,0),(0,4)\right)$.  Up to closure, all curves of genus $g$ arise as nondegenerate curves with respect to this Newton polygon, since all nonhyperelliptic curves of genus $3$ are smooth plane quartics.  Thus $\mathcal{M}_\Delta=\mathcal{M}_3^\textrm{nd}=\mathcal{M}_3$, and  $\textrm{trop}(\mathcal{M}_\Delta)=\textrm{trop}(\mathcal{M}_3^\textrm{nd})=\textrm{trop}(\mathcal{M}_3)=\mathbb{M}_3$. However, it is not the case that $\mathbb{M}_\Delta=\mathbb{M}_3$:  as computed in \cite[\S 5]{BJMS}, only about $29.5\%$ of genus $3$ metric graphs appear in $\mathbb{M}_\Delta$.  Since $\textrm{trop}(\mathcal{M}_\Delta)=\mathbb{M}_3$, we have $\mathbb{M}_\Delta\subsetneq\textrm{trop}(\mathcal{M}_\Delta)$.  A similar argument shows $\mathbb{M}_3^{\rm planar}\subsetneq \textrm{trop}(\mathcal{M}_3^\textrm{nd})$.

We can still hope for an equality of dimensions; since tropicalization preserves dimension, this is the same as asking whether  $\dim(\mathcal{M}_g^{\textrm{nd}})=\dim(\mathbb{M}_g^\textrm{planar})$, and whether $\dim(\mathcal{M}_\Delta)=\dim(\mathbb{M}_\Delta)$.  The fact that $\dim(\mathcal{M}_g^{\textrm{nd}})=\dim(\mathbb{M}_g^\textrm{planar})$ is the content of \cite[Theorem 1.1]{BJMS}.  Our Theorem \ref{Thm:MainTheorem}, to be proven in Sections \ref{section:beehive}, \ref{section:nonmaximal}, and \ref{section:hyperelliptic}, states that $\dim\left(\mathbb{M}_\Delta\right)= \dim\left(\mathcal{M}_\Delta\right)$ when $\Delta$ is either nonhyperelliptic, or maximal hyperelliptic.  We summarize the dimensional relationships below:
\begin{equation*}
 \begin{matrix}
\dim(\mathcal{M}_\Delta)  & \leq &  \dim(\mathcal{M}_g^{\rm nd}) &  \leq & \dim(\mathcal{M}_g) \\
 \text{\rotatebox{90}{$=$}} && \text{\rotatebox{90}{$=$}} & & \text{\rotatebox{90}{$=$}} \\
\dim(\mathbb{M}_\Delta)  & \leq &  \dim(\mathbb{M}_g^{\rm planar}) &  \leq & \dim(\mathbb{M}_g) \\
\end{matrix}
\end{equation*}
We now recall several dimensional results on the algebraic moduli spaces.

\begin{theorem}[\cite{Koelman}, Theorem 2.5.12]
Let $\Delta$ be a maximal nonhyperelliptic polygon, with associated toric surface $X(\Delta)$ with automorphism group $\textrm{Aut}(X(\Delta))$.  Then \[\dim(\mathcal{M}_{\Delta}) = |( \Delta\cap \mathbb{Z}^{2})| - \dim \textrm{Aut}(X(\Delta)) - 1.\]
\label{theorem:automorphisms}
\end{theorem}
The dimension of $\textrm{Aut}(X(\Delta))$ can be computed combinatorially using column vectors. Letting $c(\Delta)$ denote the number of column vectors of $\Delta$, we have the following result. 

\begin{theorem}[\cite{Semigroups}, Theorem 5.3.2] We have
$\dim Aut(X(\Delta)) = c(\Delta) + 2$.
\label{lemma:columns_automorphisms}
\end{theorem}

It follows from these two results that $\textrm{dim}(\mathcal{M}_\Delta)=|(\Delta\cap\mathbb{Z}^2)|-c(\Delta)-3$.  For instance, if $\Delta$ is the polygon from Figure \ref{figure:first_honeycomb}, we have $c(\Delta)=4$ (each edge has exactly one column vector) and $|\Delta\cap\mathbb{Z}^2|=25$, so $\textrm{dim}(\mathcal{M}_\Delta)=25-4-3=18$.

The dimension formula for a nonmaximal nonhyperelliptic polygon is more complicated.  Given such a polygon $\Delta$ of genus $g$, recall that $\Delta^{(0)}=\Delta^{(1)(-1)}$ is the unique maximal polygon of genus $g$ containing $\Delta$.  Let $A=(\Delta^{(0)}\setminus\Delta)\cap\mathbb{Z}^2$ be the set of lattice points appearing in $\Delta^{(0)}$ and not $\Delta$.  We can relate $\dim(\mathcal{M}_\Delta)$ to $\dim(\mathcal{M}_{\Delta^{(0)}})$ in terms of the rank of a matrix constructed based on the column vectors of $\Delta$.  Let $a_1,a_2,...,a_n$ be the $n$ elements of $A$ and $c_1,c_2,\ldots,c_m$ be the $m$ column vectors of $\Delta^{(0)}$.  Let $J$ be the $n\times m$ with generic entries such that the entry in the $i^{th}$ row and $j^{th}$ column is nonzero if and only if $a_i - c_j \in \Delta\cap\mathbb{Z}^2$.

\begin{theorem}[\cite{Koelman}, Theorem 2.6.12]
If $X$ is the toric variety associated to $\Delta^{(0)}$, we have  \[\dim(\mathcal{M}_{\Delta}) = |\Delta\cap\mathbb{Z}^2| - \dim(Aut(X)) + \text{rank}(J) - 1.\]
\end{theorem}
Comparing this formula to the one from Theorem \ref{theorem:automorphisms}, we have the following corollary.

\begin{cor}\label{corollary:gap_maximal_nonmaximal}
We have
\[\dim(\mathcal{M}_{\Delta}) = \dim(\mathcal{M}_{\Delta^{(0)}}) - |A| + \text{rank}(J).\]
\end{cor}


\section{Computing $\dim\left(\mathbb{M}_\mathcal{T}\right)$} \label{section:computing_dimension}

Throughout this section we will assume that $\Delta$ is a nonhyperelliptic polygon, and $\TTT$ be a regular unimodular triangulation of $\Delta$.  The main goal of this section is to provide a method to compute $\dim(\mathbb{M}_\mathcal{T})$ in terms of the combinatorial properties of $\mathcal{T}$.
Our strategy, which mirrors the proof of \cite[Lemma 4.2]{BJMS}, is as follows.  Since  $\mathbb{M}_\mathcal{T}=\kappa\circ\lambda(\Sigma(\mathcal{T}))$, we first determine $\dim\left(\lambda(\Sigma(\mathcal{T}))\right)$ in Lemma \ref{lemma:dim_after_lambda}.  We then focus on the radial edges of $\mathcal{T}$, showing that these include all edges concatenated under $\kappa$ in Lemma \ref{lemma:concatenation} and counting their number in Lemma \ref{lemma:counting_r}.  From here we characterize degrees of freedom in edge lengths depending on concatenation in Lemma \ref{lemma:single_cycle}.  This allows us to compute in Lemma \ref{lemma:kappa_bar} the rank of the map $\overline{\kappa} $, the restriction of  $\kappa $ to the linear span of $\lambda(\Sigma(\mathcal{T}))$.  This is then used to prove Theorem \ref{theorem:lower_bound_mt}, our formula for the dimension of $\mathbb{M}_\mathcal{T}$.  The  section closes by applying this result to study the dimension of  $\mathbb{M}_\Delta$.



Throughout, we will assume without loss of generality that all interior edges in $\mathcal{T}$ intersect $\Delta^{(1)}$.  If some edges do not, we may iteratively remove such triangles until we end up with a triangulation $\mathcal{T}'$ of a smaller polygon $\Delta'$, giving rise to exactly the same metric graphs as $\mathcal{T}$, so that $\mathbb{M}_\mathcal{T}=\mathbb{M}_\mathcal{T'}$. This follows from the fact that the triangles in $\mathcal{T}$ not intersecting the interior lattice points do not contribute to the skeleton; this is  the essence of the proof of \cite[Lemma 2.6]{BJMS}. Thus to determine $\dim(\mathcal{T})$, it suffices to determine $\dim(\mathcal{T}')$.   We begin with the following lemma.

\begin{lemma}\label{lemma:dim_after_lambda}
Let $\mathcal{T}$ be a regular unimodular triangulation of a lattice polygon $\Delta$ of genus $g$, and let $E$ be the set of interior edges of $\mathcal{T}$.  Then $\dim(\lambda(\Sigma(\mathcal{T})))=|E|-2g$.
\end{lemma}

\begin{proof}
Let $A=\Delta\cap\mathbb{Z}^2$, so that $\lambda$ maps $\mathbb{R}^A$ to $\mathbb{R}^E$.  Let $r$ be the number of boundary points of $\Delta$, so that $|A|=g+r$.  By Pick's Theorem, the number $T$ of triangles in  $\mathcal{T}$ is equal to $2g+r-2$.  We also have $3T=2|E|+r$, since each triangle contributes $3$ edges to $\mathcal{T}$, with interior edges double-counted.  It follows that $3(2g+r-2)=2|E|+r$, which simplifies to $|E|=3g+r-3$, or equivalently $|E|=|A|+2g-3$.  We can rewrite this as $|A|=|E|-2g+3$.

By the rank-nullity theorem, we have $\rank(\lambda)+\textrm{null}(\lambda)=|A|$.  The nullity of $\lambda$ is the dimension of the fiber over any point in $\mathbb{R}^E$; choosing such a point $p$ in $\lambda(\Sigma(\mathcal{T})^\circ)$, we may identify $p$ with a smooth plane tropical curve dual to $\mathcal{T}$, unique up to translation (since $\mathcal{T}$, edge lengths, and position in $\mathbb{R}^2$ are all the data necessary to specify a tropical curve).  We then have that $\textrm{null}(\lambda)$ is equal to the number of degrees of freedom in choosing a tropical polynomial yielding the tropical curve $p$ up to translation; there is one degree of freedom comes from scaling the coefficients, and two more degrees of freedom come from linear change of coordinates corresponding to translation. More formally, $\textrm{null}(\lambda)=\textrm{span}(v_1,v_2,v_3)$, where $v_1$ is the all-ones vector (corresponding to scaling the coefficients),   $v_2$ is the vector with $i$ in the entry corresponding to $x^i\odot y^j$ (corresponding to translation in the  $x$-direction), and $v_3$ is the vector with $j$ in the entry corresponding to $x^i\odot y^j$ (corresponding to translation in the  $y$-direction).   Thus $\textrm{null}(\lambda)=3$, and we have $\rank(\lambda)=|A|-\textrm{null}(\lambda)=|E|-2g+3-3=|E|-2g$.

Since $\Sigma(\mathcal{T})$ is a full-dimensional cone in $\mathbb{R}^A$, its image under $\lambda$ is a $(|E|-2g)$-dimensional cone.

\end{proof}

We offer the following natural interpretation of where the $2g$ linear equations cutting down the dimension of $\lambda(\Sigma(\mathcal{T}))$ are coming from. Each point $P\in \Delta^{(1)}\cap\mathbb{Z}^2$ corresponds to a cycle bounding some face of $C$. The lengths on the edges of such a cycle are constrained by inequalities ensuring each length is positive, along with  two linear equations; these are exactly the conditions such that the edges do indeed form a closed loop. These equations are determined by the primitive vectors parallel to the $1$-dimensional faces of $\mathcal{T}$ containing $P$. Indeed, for a lattice point $P\in \Delta^{(1)}$, let $\nu_1,\ldots \nu_n$ be the primitive vectors beginning at $P$ in the direction of the one-dimensional faces (that is, edges) in $\mathcal{T}$ including $P$.  By abuse of notation will refer to the faces of $\TTT$ and the vectors both as $\nu_i$. Then let $\mu_i$ be obtained by rotating $\nu_i$ by $\frac{\pi}{2}$. For any tropical curve $C$ corresponding to a point in $\lambda(\Sigma(\TTT))$ with edges $e_i$ dual to each $\nu_i$ of lengths $\ell_i$, we must have that
\begin{equation*} \label{ConstEq}
	\sum\limits_{i=1}^{n}\ell_i \mu_i=\left<0,0\right>.
\end{equation*}
This yields two linear equalities, one for each coordinate.  This is illustrated in Figure \ref{fig:Eqfig}.  Since there are $g$ lattice points, this yields the $2g$ linear equations.

	\begin{figure}[hbt] 
    	\centering
		\includegraphics[scale=0.5]{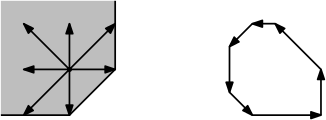}
        \caption{The $\nu_i$ emanating from $P$ in $\Delta$, and the loop dual to them}
        \label{fig:Eqfig}
	\end{figure}

The next step is to understand the dimension of our cone when we then apply $\kappa:\mathbb{R}^E\rightarrow\mathbb{R}^{3g-3}$.  This will require a careful consideration of which edges in $\Gamma$ are concatenated under $\kappa$.  We will see that certain faces in $\mathcal{T}$ play a key role.  We say a one-dimensional face $e$ of $\TTT$ is a \textit{radial face} if one of the endpoints is in $\partial \Delta$, the other is $\partial \Delta^{(1)}$, and the interior of $e$ is contained in $\Delta \setminus \Delta^{(1)}$.  The following lemma and proposition make precise why we are concerned with such faces.

\begin{lemma}\label{lemma:concatenation}
Let $e_1$ and $e_2$ be two adjacent edges of a tropical curve $C$ in $\lambda(\Sigma(\TTT))$. Then $e_1$ and $e_2$ concatenate into one edge under $\kappa$ if and only if the one-dimensional faces dual to $e_1$ and $e_2$ are adjacent radial edges.
\end{lemma}

\begin{proof}
The reverse direction is clear. Conversely, assume $e_1$ and $e_2$ concatenate into one edge under $\kappa$. We know the faces $f_1$ and $f_2$ dual to $e_1$ and $e_2$ respectively are contained in some unimodular triangle $T$. Let $f_3$ be the third face of $T$ and $P$ be the intersection point of $f_1$ and $f_2$. First we will argue that $f_3$ is not a split.  If an edge contributing to the skeleton is dual to a nontrivial split, that split must be part of two triangles, each of which contains an interior lattice point.  Thus the bridge dual to the split must be connected to two edges that form a bounded cycle, and are thus themselves not bridges. This means that bridges will not concatenate with other edges under $\kappa$, so by assumption $f_3$ must not be a split.

 Without loss of generality we may assume $T$ is the triangle with vertices $(1,0)$, $(0,0)$, and $(0,1)$, where $P$ is the point $(0,1)$ and that none of $\Delta$ is contained strictly below the $x$-axis as $f_3$ is not a split.  But then if $f_1$ and $f_2$ are not radial $\Delta^{(1)}$ must intersect the interior of $T$ and since $\Delta$ and $\Delta^{(1)}$ are convex there must be a lattice point of $\Delta^{(1)}$ contained in $\{(x,y)| 0\leq y<1| \}$;  this is clearly impossible, as illustrated is illustrated in Figure \ref{RadialLemma}.
\end{proof}

	\begin{figure}[hbt] 
    	\centering \includegraphics[scale=0.6]{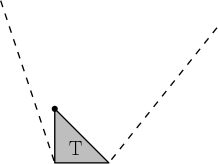}
        \caption{The triangle $T$ from Lemma \ref{lemma:concatenation}}
        \label{RadialLemma}
      \end{figure}


Since radial edges will play a key role, we prove the following lemma that counts how many such edges there are.
\begin{lemma}\label{lemma:counting_r}
Let $\mathcal{T}$ be a unimodular triangulation of $\Delta$, where all edges intersect $\Delta^{(1)}$.  Let $g$ and $g^{(1)}$ be the genus of $\Delta$ and $\Delta^{(1)}$, respectively, and let $r$ be the number of boundary points of $\Delta$.  Then the number of radial edges in $\mathcal{T}$ is $g+r-g^{(1)}$.
\end{lemma}

\begin{proof}
Delete all non-boundary edges in $\mathcal{T}$ that are not radial edges, and add in any unimodular triangulation of $\Delta^{(1)}$, including $\partial \Delta^{(1)}$.  We claim that the resulting subdivision $\mathcal{T}'$ of $\Delta$ is a  unimodular triangulation:  any polygon larger than a primitive triangle would be contained in $\overline{\Delta\setminus\Delta^{(1)}}$, but no such polygon can have a lattice point separated from another by an edge of $\mathcal{T}$ unless that edge did not intersect $\Delta^{(1)}$. Note that $\mathcal{T}'$ also has the same number of radial edges as $\mathcal{T}$, since no radial edges have been removed or added.

Within $\mathcal{T}'$, let $T_1$ be the number of triangles in $\Delta^{(1)}$, and let $T_2$ be the remaining triangles.  Letting $r^{(1)}$ be the number of lattice boundary points in $\Delta^{(1)}$, we have by Pick's Theorem that
\[T_1+T_2=2g+r-2\]
and
\[T_1=2g^{(1)}+r^{(1)}-2.\]
Thus we have
\[T_2=2g+r-2g^{(1)}-r^{(1)}.\]
Let $R$ denote the number of radial edges in $\mathcal{T}$ (and thus in $\mathcal{T}'$). Each triangle contributing to $T_2$ has $3$ edges, giving us $3T_2$ edges if we double-count shared edges.  The number of edges that are counted once is $r+r^{(1)}$, and the number of edges being counted twice is $R$.  Thus we have
\[3T_2=r+r^{(1)}+2R.\]
It follows that
\[6g+3r-6g^{(1)}-3r^{(1)}=r+r^{(1)}+2R,\]
which simplies to 
\[2R=6g+2r-6g^{(1)}-4r^{(1)},\]
or
\[R=3g+r-3g^{(1)}-2r^{(1)}.\]
Note that $g^{(1)}+r^{(1)}=g$, so we may simplify this to 
\[R=g+r-g^{(1)},\]
as desired.
\end{proof}

Our next result considers how many information is lost when we concatenate certain edges in a tropical cycle.

\begin{lemma}\label{lemma:single_cycle}  
Let $C$ be a convex piece-wise linear simple closed curve with rational slopes, with edges $e_1,\cdots,e_n$ with lattice lengths $\ell_1,\cdots,\ell_n$, and primitive normal vectors $\nu_1,\cdots,\nu_n$; and assume further that for some $m\geq 2$, the lengths $\ell_1,\cdots,\ell_m$ are unknown but their sum $\ell=\ell_1,\cdots,\ell_m$ is known.  Let $L$ be the set of all possible $m$-tuples of lengths $\ell_1,\cdots,\ell_m$.
\begin{itemize}
    \item If the endpoints of $\nu_1,\cdots,\nu_m$ are all collinear, then $\dim(L)=m-2$.
     \item If the endpoints of $\nu_1,\cdots,\nu_m$ are not all collinear, then $\dim(L)=m-3$.
\end{itemize}
\end{lemma}

\begin{proof}Given the normal vectors $\nu_1,\cdots,\nu_n$, we have $n-2$ degrees of freedom in choosing the lattice lengths of $C$, with the last $2$ lengths being determined by their slopes and the previous set of length choices.  This means that given $\ell_{m+1},\cdots,\ell_{n}$, there are $m-2$ degrees of freedom in choosing the remaining $m$ edge lengths.  There is also the added condition that these edges must add to $\ell$.  We will see that this does not actually add a constraint if the endpoints of $\nu_1,\cdots,\nu_m$ are collinear, but that it does if they are not all collinear.

First assume the endpoints of $\nu_1,\cdots,\nu_m$ are collinear.  After a change of coordinates, we may assume that they all lie on a horizontal line.  The dual edges $e_1,\cdots,e_m$ then each have lattice length equal to their horizontal width.  Letting $(a,b)$ and $(c,d)$ denote the start and end of the portion of the cycle consisting of edges $e_{m+1},\cdots,e_n$, we see that $\ell_1+\cdots+\ell_m=|a-c|$.  Thus the prescribed length $\ell$ will automatically be achieved regardless of our choice of $\ell_1,\cdots,\ell_n$, meaning that $\dim(L)=m-2$.

Now assume the endpoints of $\nu_1,\cdots,\nu_m$ are not collinear.  We will show that the sum $\ell_1+\cdots+\ell_m$ is not determined by the other edge lengths, meaning we need the additional linear constraint that $\ell_1+\cdots+\ell_m=\ell$.  Choose $i$ so that $\nu_{i-1}$, $\nu_{i}$, and $\nu_{i+1}$ do not have collinear endpoints.  Without loss of generality, we may assume that $\nu_{i-1}=\langle1,a\rangle$  where $a\leq 1$, $\nu_{i}=\langle1,0\rangle$, and $\nu_{i+1}=\langle b,c\rangle$ where $b\neq 1$ (since the three endpoints are not collinear) and $c>0$ (since the polygon is convex). The lengths of $e_{i-1}$ and $e_{i}$ are equal to their vertical heights since both $\nu_{i-1}$ and $\nu_{i}$ have second coordinate equal to $1$, while the same does not hold for $e_{i+1}$.  Suppose we are given a valid $n$-tuple of lengths $(\ell_1,\ldots,\ell_n)$, achieved by edges $(e_1,\cdots,e_n)$  
Write the endpoints of $e_i$ as $(x,y_0)=e_i\cap e_{i-1}$ and $(x,y_1)=e_i\cap e_{i+1}$, where $y_0<y_1$.  Choose $\varepsilon>0$ so that $y_0+\varepsilon<y_1-\frac{b}{c}\varepsilon$, and build a new cycle so that $e_i'$  now has endpoints $(x+\varepsilon,y_0+\varepsilon)$ and $(x+\varepsilon,y_1-\frac{b}{c}\varepsilon)$; this corresponds to slightly increasing the lengths $\ell_{i-1}$ and $\ell_{i+1}$ to $\ell_{i-1}'$ and $\ell_{i+1}'$, while also changing $\ell_i$ to $\ell_i'$ if necessary.  We claim that the new cycle gives a different value for $\ell$.  To see this, note that $\ell_{i-1}'=\ell_{i-1}+\varepsilon$ and $\ell_{i+1}'=\ell_{i+1}+\frac{\varepsilon}{c}$, while $\ell_i'=\ell_i-(1+\frac{b}{c})\varepsilon$.  We thus have $\ell_{i-1}'+\ell_i'+\ell_{i+1}'=\ell_{i-1}+\ell_{i}+\ell_{i+1}+(1+\frac{1}{c}-1-\frac{b}{c})\varepsilon=\ell_{i-1}+\ell_{i}+\ell_{i+1}+\frac{1-b}{c}\varepsilon$.  Since $b\neq 1$, this difference between $\ell_{i-1}'+\ell_i'+\ell_{i+1}'$ and $\ell_{i-1}+\ell_i+\ell_{i+1}$ is nonzero; as all other lengths were unchanged, we do obtain different values of $\ell$ as claimed.  Thus the condition of a prescribed $\ell$ does add a constraint, meaning that $\dim(L)=m-3$.
\end{proof}

Before we state the following lemma, we establish some notation for our triangulation $\mathcal{T}$.  Every point $P$ in $\partial \Delta^{(1)}$ falls into exactly one of three categories. Letting $\nu_1,\ldots,\nu_m$ be the consecutive radial faces incident to $P$, we either have:
\begin{enumerate}
    \item $m=1$;
    \item $m\geq 2$, with all endpoints of $\nu_1,\ldots,\nu_m$ collinear; or
    \item $m\geq 3$, with not all endpoints of $\nu_1,\ldots,\nu_m$ collinear.
\end{enumerate}  We refer to the points of $\partial \Delta^{(1)}$ satisfying these properties as Type $1$, Type $2$, and Type $3$, respectively, and we let $b_i$ denote the number of points of Type $i$.

\begin{lemma}\label{lemma:kappa_bar} Let $S$ denote the $(|E|-2g)$-dimensional subspace in $\mathbb{R}^E$ containing $\lambda(\Sigma(\mathcal{T}))$, and let $\overline{\kappa}$ be the restriction of $\kappa$ to this subspace.
Then the map $\overline{\kappa}:S\rightarrow\mathbb{R}^{3g-3}$ has rank $g-3+b_2+2b_3$.
\end{lemma}

\begin{proof}
Recall that we assumed at the start of this section that $\mathcal{T}$ only has interior edges that intersect $\Delta^{(1)}$; this means that $\overline{\kappa}$ does not have to delete any bounded edges, and only has to concatenate some.

From the proof of Lemma \ref{lemma:dim_after_lambda}, we have $|E|=3g+r-3$.  This means that $\textrm{rank}(\overline{\kappa})+\textrm{null}(\overline{\kappa})=|E|-2g=3g+r-3-2g=g+r-3$, or $\textrm{rank}(\overline{\kappa})=g+r-3-\textrm{null}(\overline{\kappa})$.  We will compute the nullity of $\overline{\kappa}$ by determining the dimension of $\overline{\kappa}^{-1}(p)$ for an arbitrary point $p=(\ell_1,\ldots,\ell_{3g-3})\in\overline{\kappa}(\lambda(\Sigma(\mathcal{T})^\circ))$.

By our assumptions on $\mathcal{T}$, the map $\overline{\kappa}$ does not delete any edges, it only concatenates them.  If an edge is not concatenated under $\kappa$, then that coordinate in $\mathbb{R}^E$ can be identified with that coordinate in $\mathbb{R}^{3g-3}$, and every point in $\overline{\kappa}^{-1}(p)$ will have the same value $\ell_i$ in that coordinate.  Thus it is only coordinates in $\mathbb{R}^{3g-3}$ made from concatenating edges (that is, adding coordinates in  $\mathbb{R}^{E}$) that can contribute to the dimension of $\overline{\kappa}^{-1}(p)$.  By Lemma \ref{lemma:concatenation}, the only edges that are concatenated are dual to sequences of radial edges emanating from the same lattice point of $\partial \Delta^{(1)}$.  Since these coordinates in $\mathbb{R}^E$ only correspond to one interior lattice point, we may separately consider the contribution of each boundary point $P$ of $\Delta^{(1)}$ to $\dim(\overline{\kappa}^{-1}(p))$.

If $P$ is on the boundary of $\Delta^{(1)}$, then let $\nu_1,\cdots,\nu_m$ be the consecutive radial faces incident to $P$.  The contribution of $P$ to $\dim(\overline{\kappa}^{-1}(p))$ is then determined by what type of point $P$ is; we claim that the contribution is
\begin{itemize}
    \item $0$ for Type 1 (that is, if $m=1$);
    \item $m-2$ for Type 2 (that is, if $m\geq 2$ and the endpoints of $\nu_1,\cdots,\nu_m$ are all collinear); and
    \item $m-3$ for Type 3 (that is, if $m\geq 3$ and the endpoints of $\nu_1,\cdots,\nu_m$ are not all collinear).
\end{itemize}
We argue this as follows.  For Type 1, no concatenation occurs, so there is no contribution to $\dim(\overline{\kappa}^{-1}(p))$.  For Type 2, we have lengths $l_1,\cdots,l_m$ adding up to some length $\ell$.  As shown in Lemma \ref{lemma:single_cycle}, since the endpoints of the $\tau_i$'s are collinear, there are two linear equations governing the possible values of $l_1,\cdots,l_m$, namely those that ensure that the cycle they are a part of is a closed loop.  (The condition that $l_1+\cdots+l_m=\ell$ is already determined by the other edges of the cycle.)  Thus there are $m-2$ degrees of freedom in choosing the $m$ lengths.  The same holds for Type 3, except that the additional constraint $l_1+\cdots+l_m=\ell$ does not automatically hold, meaning there are $m-3$ degrees of freedom in choosing the lengths $l_1,\cdots,l_m$.

Adding up all these contributions, we have a contribution of $1$ for every edge connecting $\partial\Delta$ to $\partial\Delta^{(1)}$, minus $1$ for every point of Type 1, minus $2$ for every point of Type 2, and minus $3$ for every point of Type 3. Letting $R$ denote the total number of radial edges, we thus have have
\[\textrm{null}(\overline{\kappa})=R-b_1-2b_2-3b_3.\]
We know from Lemma \ref{lemma:counting_r} that $R=g+r-g^{(1)}$, or equivalently that $R=b_1+b_2+b_3+r$.  It follows that
\[\textrm{null}(\overline{\kappa})=b_1+b_2+b_3+r-b_1-2b_2-3b_3=r-b_2-2b_3.\]
We conclude that\[\rank(\overline{\kappa})=g+r-3-(r-b_2-2b_3)=g-3+b_2+2b_3.\]
\end{proof}

This allows us to prove Theorem \ref{theorem:lower_bound_mt}, which states that
\[\dim(\mathbb{M}_\mathcal{T})=\dim(\kappa \circ \lambda(\Sigma(\TTT)))= 3g-3-2g^{(1)}-2b_1-b_2.\]

\begin{proof}[Proof of Theorem \ref{theorem:lower_bound_mt}]
First note that $\kappa \circ \lambda(\Sigma(\TTT))=\overline{\kappa} \circ \lambda(\Sigma(\TTT))$, so we may instead show 
\[\dim(\overline{\kappa} \circ \lambda(\Sigma(\TTT)))= 3g-3-2g^{(1)}-2b_1-b_2.\]
Since $\lambda(\Sigma(\TTT)))$ is a full-dimensional cone inside of $S$, and since $\overline{\kappa}:S\rightarrow\mathbb{R}^{3g-3}$ has rank $g-3+b_2+2b_3$, we have  $\dim(\overline{\kappa}\circ \lambda(\Sigma(\mathcal{T})))=g-3+b_2+2b_3$.  Writing this to highlight the codimension, we have
\[\dim(\overline{\kappa}\circ \lambda(\Sigma(\mathcal{T})))=g-3+b_2+2b_3=3g-3+b_2+2b_3-2g,\]
and we take advantage of the fact that $g=g^{(1)}+b_1+b_2+b_3$ to rewrite this as
\[\dim(\overline{\kappa}\circ \lambda(\Sigma(\mathcal{T})))=3g-3+b_2+2b_3-2g^{(1)}-2b_1-2b_2-2b_3=3g-3-2g^{(1)}-2b_1-b_2.\]
\end{proof}

This immediately gives us the following corollary.

\begin{cor}\label{corollary:delta_dimension_max}
The dimension of $\mathbb{M}_\Delta$ is the maximum of
\[3g-3-2g^{(1)}-2b_1-b_2,\]
or equivalently
\[g-3+b_2+2b_3,\]
taken over all regular unimodular triangulations of $\Delta$.
\end{cor}

In the next section we will consider how to find regular unimodular triangulations achieving the maximum.  However, for certain polygons it is easy to verify that a maximum has been obtained.  Note that a lattice boundary point of $\Delta^{(1)}$ can only be of Type 3 if it is a vertex:  if it is on the relative interior of a face $\tau$ of $\Delta^{(1)}$, a radial edge can only connect that point to a lattice point on the relaxed face $\tau^{(-1)}$ by convexity.  Thus if $b$ is the number of non-vertex boundary lattice points of $\Delta^{(1)}$, in the best case scenario we have that all $b$ such points are of Type 2 and all vertices of $\Delta^{(1)}$ are of Type 3.  This means for any polygon $\Delta$ we have that
\[\dim(\mathbb{M}_\Delta)\leq3g-3-2g^{(1)}-b.\]
If we can find a regular unimodular triangulation with these optimal properties, then we will have that the above formula is an equality.  For example, any honeycomb triangulation satisfies these properties, meaning that if $\Delta$ is a honeycomb polygon, then
\[\dim(\mathbb{M}_\Delta)=3g-3-2g^{(1)}-b.\]
For a non-honeycomb example, consider the triangulated polygon $\Delta$ appearing in Figure \ref{figure:beehive_example}, along with a dual tropical curve as a witness to its regularity.  All vertices of $\Delta^{(1)}$ are of Type 3, and all nonvertex boundary points of $\Delta^{(1)}$ are of Type $2$.  This is the best possible scenario, so the dimension of $\mathbb{M}_\Delta$ is $3g-3-2g^{(1)}-b=3\cdot 7-3-2\cdot1-1=15$.

	\begin{figure}[hbt]
   		 \centering
        \includegraphics[scale=1]{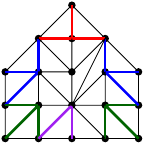}\qquad\includegraphics[scale=0.5]{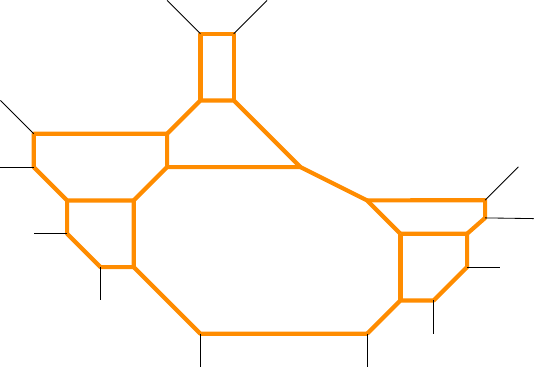}
		\caption{A regular triangulation maximizing dimension, and a dual tropical curve}
		\label{figure:beehive_example}
	\end{figure}
	
We now discuss in general the types of triangulations that maximize the value of $b_2+2b_3$; we leave to the next section the consideration of whether or not such triangulations can be chosen to be regular.  Assume for the moment that $\Delta$ is a maximal polygon, so that $\Delta=\Delta^{(-1)(1)}$.  Let $V(\Delta^{(1)})=\{v_1,\ldots v_n\}$ be ordered cyclically. Let the 1-dimensional faces of $\Delta^{(1)}$ be $\{\tau_1,\ldots \tau_n\}$ where $\tau_i$ has endpoints $v_i$ and $v_{i+1}$ (we work with the indices modulo $n$). We say that a unimodular triangulation $\mathcal{T}$ of $\Delta$ is a \emph{beehive triangulation} of $\Delta$ if
\begin{enumerate}
    \item $\mathcal{T}$ includes all boundary edges of $\Delta^{(1)}$;
    \item $v_i$ is connected to $v_i^{(-1)}$ for all $i$; and
    \item for each $i$, the number of lattice points on $\tau_i$ connected to at least two lattice points on $\tau_i^{(-1)}$ is maximized.
\end{enumerate}

Two examples of beehive triangulations of maximal polygons are illustrated in Figure \ref{figure:beehive_triangulations}, with the interior polygons shaded as any unimodular completion will preserve beehive-ness.  (The third triangulation is a beehive triangulation of a nonmaximal polygon, which will we define shortly.)

	\begin{figure}[hbt] 
    	\centering
		\includegraphics[scale=1]{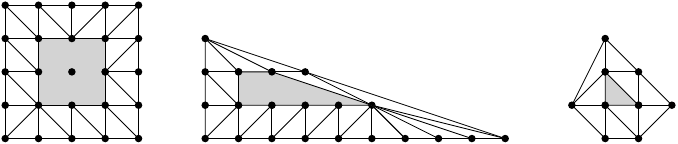}
        \caption{Beehive triangulations, with interior polygons shaded}
        \label{figure:beehive_triangulations}
	\end{figure}

\begin{lemma}\label{lemma:beehive_best}
Any beehive triangulation achieves the maximum possible value of $b_2+2b_3$.
\end{lemma}

Before we prove this lemma, we will consider how we can extend the definition of beehive to nonmaximal polygons. First we replace condition (2) with $v_i$ being connected to the points of $\tau_{i-1}^{(-1)}\cap \Delta$ and $\tau_i^{(-1)}\cap \Delta$ closest to $v_i^{(-1)}$. For condition (3), it is no longer the case that we may treat each pair $\tau_i^{(-1)}$ and $\tau_j^{(-1)}$ independently for the purposes of achieving the maximum possible value of $b_2+2b_3$. For instance, in the rightmost triangulation in Figure \ref{figure:beehive_triangulations}, the fact that one interior lattice point is connected to both lattice points of the bottom-most edge prevents another lattice point from being connected to more than one such lattice point; and if we had flipped the diagonal edge to prioritize the other lattice point, we would have achieved a lower value of $b_2+2b_3$.  Thus we replace condition (3) with the more opaque requirement that we complete the triangulation so as to maximize $b_2+2b_3$.

\begin{proof}  First assume $\Delta$ is maximal.  Certainly there is no harm in connecting $v_i$ to $v_i^{(-1)}$: the only edge in a triangulation that could separate them would connect $\tau_{i-1}^{(-1)}$ to $\tau_i^{(-1)}$, which does not improve the type of any interior lattice points.  There is also no harm in including all boundary egdes of $\Delta^{(1)}$, since this will not block any possible radial edges.

At this point all we need to do is determine, for each $i$, which lattice points on $\tau_i$ to connect to which lattice points on $\tau_i^{(-1)}$. Certainly each lattice point $u$ of $\tau_i$ will be connected to at least one lattice point of $\tau_i^{(-1)}$.  We claim that each $u$ connected to at least two lattice points of $\tau_i^{(-1)}$ will contribute exactly $1$ to $b_2+2b_3$, and that this is the maximal such contribution.  To see this, note that a nonvertex boundary point $u$ can at best be a Type 2 point, which occurs if and only if it is connected to at least two lattice points of $\tau_i^{(-1)}$; and that a vertex boundary point $v_i$ (or $v_{i+1}$) will be upgraded from a Type 1 to a Type 2 or from a Type 2 to a Type 3 (depending on what's happening in $\tau_{i-1}^{(-1)}$) if and only if it is connected to an additional lattice point of $\tau_i^{(-1)}$ besides $v_i^{(-1)}$.  In all these cases, a contribution of exactly $1$ occurs, and we can do no better by making different choices for the edges connecting $\tau_i$ and $\tau_i^{(-1)}$.  

A similar argument holds in the case where $\Delta$ is not maximal:  the prescribed edges from (1) and (2) do not interfere with any radial edges, and from there (3) maximizes the possible contributions to $b_2+2b_3$.
\end{proof}

Once we show that we can find regular beehive triangulations, we will know that they achieve the maximum possible dimension of $\dim(\MM_\TTT)$ for a given $\Delta$.  In this way, they play the same role as honeycomb triangulations for general polygons, whence the name ``beehive''.  It is not true that all honeycomb triangulations are beehive triangulations; however, it becomes true once we slice off any corners of the honeycomb polygon that do not contribute to the skeleton in the honeycomb triangulation.

\section{Maximal nonhyperelliptic polygons}
\label{section:beehive}

Let $\Delta$ be a maximal nonhyperelliptic polygon.  Our main goal of this section is to show that $\dim(\mathbb{M}_{\Delta})=\dim(\mathcal{M}_{\Delta})$.  We know by Corollary \ref{corollary:delta_dimension_max} that $\dim(\mathbb{M}_{\Delta})$ is the maximum value of $g-3+b_2+2b_3$, taken over all regular unimodular triangulations of $\Delta$.  We will show in Proposition \ref{prop:regular_unimodular} that there exists a regular beehive triangulation of $\Delta$, which will achieve the maximum possible value of $b_2+2b_3$ by Lemma \ref{lemma:beehive_best}; the key tool for this is Proposition \ref{prop:maximal_refinement}, which shows that we may preserve regularity after starting with a certain coarse subdivision.  Once we know that there exists a regular beehive triangulation, we  then determine the value of $b_2+2b_3$ in a beehive triangulation  in Proposition \ref{prop:maybe_not_regular}.  This allows us to prove Theorem \ref{theorem:lower_bound_mdelta}, which combined with the formulas from Subsection \ref{subsection:ag} gives us the desired equality of dimensions.  As an application we characterize all polygons $\Delta$ yielding dimension $2g+1$ in Theorem \ref{theorem:2g+1}.

As in the previous section, let $V(\Delta^{(1)})=\{v_1,\ldots v_n\}$ be ordered counterclockwise, and let the one-dimensional faces of $\Delta^{(1)}$ be $\{\tau_1,\ldots \tau_n\}$ where $\tau_i$ has endpoints $v_i$ and $v_{i+1}$ (treating the indices modulo $n$).   We will construct a regular beehive triangulation by first subdividing $\Delta$ into $\Delta^{(1)}$ and a collection of polygons with lattice width $1$, and then refining our subdivision from there.  We state the following three lemmas, which will be helpful in the refinement.

\begin{lemma}\label{threepointlemma}
For any regular subdivision $\mathcal{R}$ of a polygon $\Delta$, any set of affinely independent points $\{x_1,x_2,x_3\}$ in  $\Delta$, and any three heights $\{a,b,c\}$, there exists a height function $\omega$ such that $\omega(x_1)=a,\omega(x_2)=b,\omega(x_3)=c$ and $\omega$ induces the subdivision $\mathcal{R}$.
\end{lemma}

\begin{proof}
This lemma in the case of $a=b=c=0$ is the content of \cite[Exercise 2.1]{triangulations}.  Given a height function $\omega'$ with $\omega'(x_1)=\omega'(x_2)=\omega'(x_3)=0$ inducing $\mathcal{R}$, we can add an affine function $\omega''$ with $\omega''(x_1)=a$, $\omega''(x_2)=b$, and $\omega''(x_3)=c$.  We then have that $\omega:=\omega'+\omega''$ induces the same subdivision, and satisfies $\omega(x_1)=a,\omega(x_2)=b$, and $\omega(x_3)=c$.
\end{proof}

\begin{lemma}[\cite{triangulations}, Proposition 2.3.16]\label{subdivision}
Let $\mathcal{R}$ be a regular subdivision of $\Delta$ and let $\omega$ be a height function for $\Delta$. Then the following is a regular refinement of $\mathcal{R}$:
\[\mathcal{R}_{\omega}=\bigcup_{C\in \mathcal{R}}\mathcal{R}(\Delta|_{C},\omega|_{C}),\]
where $\mathcal{R}(\Delta|_{C}, \omega|_{C})$ is the subdivision of $C$ given by $\omega|_{C}$.
\end{lemma}

\begin{lemma}[\cite{regulartriangulations}, Lemma 3.3]\label{latticewidthone}
If $\Delta$ has lattice width $1$, then any subdivision of $\Delta$ is regular.
\end{lemma}

Suppose for the moment that $\Delta$ has at least one edge with $3$ distinct lattice points\footnote{It is possible that this holds for any maximal nonhyperelliptic polygon, perhaps through a proof similar to that of \cite[Lemma 1(c)]{pushingout}; however, it will be easy enough to separately prove our result in the case that no such edge exists.}; choose the labelling of the vertices and edges of $\Delta^{(1)}$ so that $\tau_n^{(-1)}$ is such an edge.  Let $u$ and $v$ be the points on $\tau_n^{(-1)}$ nearest $v_n^{(-1)}$ and $v_1^{(-1)}$, respectively; these are guaranteed to exist and to be distinct from $v_1^{(-1)}$ and $v_n^{(-1)}$ (though not necessarily from each other).  Let $\TTT_0$ be the subdivision of $\Delta$ given by the following height function\footnote{There is no significance whatsoever to the choice of $\frac{6}{\pi^2}$; any value in $(0,1)$ is suitable.}:
 \[\omega_0(p)= \begin{cases} 
  \dfrac{6}{\pi^2} & \textrm{if } p \in \ \conv(\{u,v \})\cap \ZZ^2 \\
      1 &  \textrm{if $p$ is another boundary point} \\
      0 & \textrm{otherwise.}
   \end{cases}
\]
Let $C_i=\conv(\tau_i,\tau_i^{(-1)})$ for $i=1,2,\ldots n$. The subdivision $\TTT_0$ has the $n$ cells $\Delta^{(1)}$, $C_1,\ldots,C_{n-1}$, as well as three more cells subdividing $C_{n}$: a pair of unimodular triangles $T_1=\conv(v_1,v_1^{-1},v)$ and $T_2=\conv(v_n,v_n^{-1},u)$ bordering $C_{1}$ and $C_{n-1}$, respectively, and an intermediate cell $C_{n}'$.  One can see an example of the subdivision in Figure \ref{fig:exampleofmaxconstruction}.

\begin{figure}[hbt]\label{fig:exampleofmaxconstruction}
\centering
\includegraphics[scale=1]{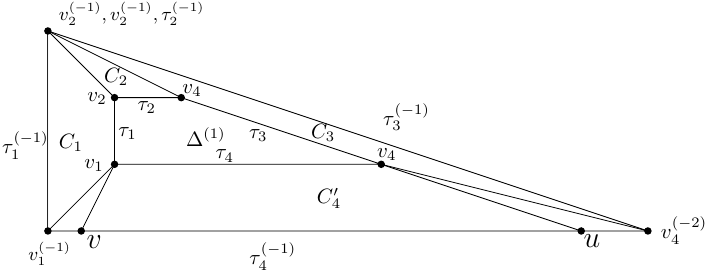}
\caption{A maximal polygon $\Delta$ with the cells of $\TTT_0$ labeled}
\end{figure}

\begin{prop}\label{prop:maximal_refinement}
For any refinements $\mathcal{R}_1, \ldots, \mathcal{R}_{n-1}, \mathcal{R}_n$ of the faces $C_1,\ldots,C_{n-1},C_n'$, there exists a regular refinement $\mathcal{R}$ of $\TTT_0$ such that $\mathcal{R}|_{C_i}=\mathcal{R}_i$ for every $i\in \{1, \ldots n-1 \}$ and $\mathcal{R}|_{C_n'}=\mathcal{R}_n$
\end{prop}

\begin{proof}
We will construct a height vector $\omega_i$ for every face $C_i$ such that $\mathcal{R}_i=\mathcal{R}(\Delta|_{C_i}, \omega_i)$, and such that $\omega_i|_{C_i\cap C_j}=\omega_j|_{C_i\cap C_j}$. This will be used to create a height function $\omega$ satisfying the conditions of Lemma \ref{subdivision}.

Note that any cell $C_i$ has lattice width $1$, and so by Lemma \ref{latticewidthone} we have that the refinement $\mathcal{R}_i$ is regular.  We will inductively choose our height vectors $\omega_i$ on $C_1,\ldots,C_{n-1},C_n'$, from $\omega_1$ to $\omega_n$, compatible on overlaps.  Choose any height vector for $\omega_1$ inducing the subdivision $\mathcal{R}_1$ of $C_1$. Now assume that $\omega_i$ has been chosen for all $i\leq k$, where $k\leq n-1$. First consider the case where $k \neq n-1$. The two faces $C_k$ and $C_{k+1}$ intersect in  two lattice points (namely  $v_{k+1}$ and $v_{k+1}^{(-1)}$), and so $\omega_{k+1}$ can be chosen to agree with $\omega_k$ by Lemma \ref{threepointlemma} while inducing $\mathcal{R}_{k+1}$ on $C_{k+1}$.

Now consider the case where $k=n-1$. 
Due to the presence of the triangles $T_1$ and $T_2$ in $\TTT_0$, $C_n'$ only intersects $C_1$ at $v_1$ and $C_{n-1}$ at $v_{n-1}$.  Thus $\omega_n$ can be chosen to agree with both $\omega_1$ and $\omega_{n-1}$ by Lemma \ref{threepointlemma} while inducing $\mathcal{R}_{n}$ on $C_{n}'$

Now define $\omega$ so that it agrees with each $\omega_i$, and is defined in any way on the other lattice points of $\Delta$.  By Lemma \ref{subdivision}, the refinement
\[\mathcal{R}=\bigcup_{C\in \mathcal{T}_0}\mathcal{R}(\Delta|_C,\omega|_C)\]
is regular, and satisfies $\mathcal{R}|_{C_i}=\mathcal{R}(\Delta|_{C_i},\omega|_{C_i})=\mathcal{R}(\Delta|_{C_i},\omega_i)=\mathcal{R}_i$ for $1\leq i\leq n-1$ and $\mathcal{R}|_{C_n'}=\mathcal{R}(\Delta|_{C_n'},\omega|_{C_n'})=\mathcal{R}(\Delta|_{C_n'},\omega_n)=\mathcal{R}_i$, as desired.
\end{proof}

Since $\TTT_0$ is the start of a beehive triangulation, this allows us to prove the following result.

\begin{prop}\label{prop:regular_unimodular}
Let $\Delta$ be a maximal nonhyperelliptic polygon.  There exists a regular unimodular beehive triangulation $\mathcal{T}$ of $\Delta$.
\end{prop}

\begin{proof}
If $\Delta$ has at least one edge with $3$ lattice points, we may apply Proposition \ref{prop:maximal_refinement},
choosing refinements of $C_1,\ldots,C_{n-1},C_n'$ in $\TTT_0$ so that $C_1,\ldots,C_n$ are triangulated to satisfy the beehive conditions.  Choose $\mathcal{T}$ to be a regular unimodular triangulation that refines $\mathcal{R}$.  
Then $\mathcal{T}$ is a regular beehive triangulation of $\Delta$.

If $\Delta$ has no such edge, then every edge of $\Delta$ has exactly two lattice points.  Instead of $\TTT_0$, start with the regular subdivision induced by $\omega$ with $\omega(p)=0$ for $p\in \Delta^{(1)}$, and $\omega(p)=1$ for $p\in\partial \Delta$.  The induced subdivision $\mathcal{R}$ then has cells $\Delta^{(1)}$, $C_1,\ldots,C_n$. Let $\mathcal{T}$ be any regular unimodular refinement of $\mathcal{R}$.  We claim that $\mathcal{T}$ is beehive.  For each $i$, any refinement of $C_i$ will yield either $0$ or $1$ lattice points of $\tau_i$ connected to two or more lattice points of $\tau_i^{(-1)}$, depending on whether $\tau_i^{(-1)}$ has $1$ or $2$ lattice points.  It follows that $\mathcal{T}$ optimizes the number of such points, and so is a regular beehive triangulation of $\Delta$.
\end{proof}

We will now determine the value of $b_2+2b_3$ in a beehive triangulation.

\begin{prop}\label{prop:maybe_not_regular}
Let $\Delta$ be a maximal nonhyperelliptic polygon with $r$ boundary points and $c(\Delta)$ column vectors, and let $\mathcal{T}$ be a  unimodular beehive triangulation.  The value of $b_2+2b_3$ is $r-c(\Delta)$.
\end{prop}

\begin{proof}

Label the lattice points of $\tau_i$ as $v_i,u_1,\cdots,u_{\ell-1},v_{i+1}$, and the lattice points of $\tau_i^{(-1)}$ as $\nu_0=v_i^{(-1)}$, $\nu_1$, $\ldots$, $\nu_k=v_{i+1}^{(-1)}$.  We claim that $\TTT$ connects $\min\{|\tau_i|,|\tau_i^{(-1)}|-1\}$ of the lattice points of $\tau_i$ points to two of the lattice points of $\tau_i^{(-1)}$.

To see this, note that one maximal way to construct our beehive triangulation would be to ``zig-zag'' between $\tau_i$ and $\tau_i^{(-1)}$, connecting in sequence $\nu_0,v_i,\nu_1,u_1,\nu_2,\cdots$ and so on.  This will terminate either when we run out of lattice points on $\tau_i$ (at which point $v_{i+1}$ would be attached to any unused points of $\tau_i^{(-1)}$), or when we run out of lattice points on $\tau_i^{(-1)}$ (at which point $v_{i+1}^{(-1)}$ would be attached to any unused points of $\tau_i$).  In the former case, we will have successfully attached each lattice point of $\tau_i$ to two lattice points of $\tau_i^{(-1)}$.  In the latter case, our path ends with $\nu_{k-1},u_j,\nu_k$ for some $j$, and only $j$ lattice points of $\tau_i$ are connected to two boundary lattice points.  Since $u_j$ follows $\nu_{j}$, we have that $j=k-1$, which is $|\tau_i^{(-1)}|-1$.

We can now compute the value of $b_2+2b_3$: it is the sum over all $i$ of $\min\{|\tau_i|,|\tau_i^{(-1)}|-1\}$, or alternatively
\[b_2+2b_3=\sum_{i = 1}^n [|\tau^{(-1)}_i| - 1]-\sum_{i=1}^n\max\{0,|\tau^{(-1)}_i| - 1- |\tau_i|\}.\]
We make the observation that
\[\sum_{i = 1}^n [|\tau^{(-1)}_i| - 1]=r.\]
Furthermore, if $|\tau_i|<|\tau^{(-1)}_i| - 1$, then by Proposition \ref{proposition:counting_columns} 
we have $|\tau^{(-1)}_i| - 1- |\tau_i|=c_i(\Delta)$, where $c_i(\Delta)$ is the number of column vectors associated to $\tau_i^{(-1)}$.  Since we have $\sum_{i = 1}^n c_i(\Delta)=c(\Delta)$, we can conclude that the value of $b_2+2b_3$ is $r-c(\Delta)$.
\end{proof}

This allows us to prove that 

\[\dim(\mathbb{M}_\Delta)=g-3-c(\Delta)+r.\]

\begin{proof}[Proof of Theorem \ref{theorem:lower_bound_mdelta}]  We have by Corollary \ref{corollary:delta_dimension_max} that $\dim(\mathbb{M}_\Delta)$ is the maximum of
$g-3+b_2+2b_3$, taken over all regular unimodular triangulations of $\Delta$.  By Lemma \ref{lemma:beehive_best} and By Proposition \ref{prop:regular_unimodular}, we can find a regular beehive triangulation of $\Delta$ which will achieve this maximum.  By Proposition \ref{prop:maybe_not_regular}, the dimension given by this triangulation is $g-3-c(\Delta)+r$.  
\end{proof}

Combined with the formula $\dim(\mathcal{M}_\Delta)=g-3-c(\Delta)+r$ from Theorems \ref{theorem:automorphisms} and \ref{lemma:columns_automorphisms}, this immediately implies the following corollary, which is Theorem \ref{Thm:MainTheorem} in the maximal nonhyperelliptic case.

\begin{cor}\label{corollary:equality}
If $\Delta$ is a maximal nonhyperelliptic polygon, then
\[\dim(\mathcal{M}_\Delta)=\dim(\mathbb{M}_\Delta).\]
\end{cor}

One consequence of this result is that we can classify all maximal polygons $\Delta$ of genus $g$ that satisfy $\dim(\MM_\Delta)=\dim(\MM_g^{\textrm{planar}})$, which is equal to $2g+1$ for all $g\geq 4$ with $g\neq 7$.

\begin{theorem}\label{theorem:2g+1}

Let $\Delta$ be a maximal polygon of genus $g\geq 4$.  Then $\dim(\MM_\Delta)=2g+1$ if and only if $\Delta$ is equivalent to one of the following polygons:
\begin{itemize}
    \item $\conv((0,0),(0,3),((g+2)/2,0),((g+2)/2,3))$ with $g$ even.
    \item $\conv((0,0),(0,3),((g-1)/2,0),((g-3)/2,3))$ with $g$ odd.
    \item  $\conv((0,0),(0,2),(2,0),(4,2),(2,4))$ with $g=6$.
    \item $\conv((0,0),(0,2),(4,0),(4,2),(2,4))$ with $g=7$.
    \item $\conv((0,0),(0,4),(4,0),(4,2),(2,4))$ with $g=8$.
    \item $\conv((0,0),(0,2),(3,0),(5,2),(2,4),(5,4))$ with $g=10$.
\end{itemize}
In particular, for $g\geq 11$, there exists a unique maximal polygon $\Delta$ with $\dim(\mathbb{M}_\Delta)=2g+1$.
\end{theorem}

\begin{proof}
First note that if $\dim(\mathbb{M}_\Delta)=2g+1$, then $\Delta$ is nonhyperelliptic, since for $\Delta$ hyperelliptic we have $\dim(\mathbb{M}_\Delta)\leq \dim(\mathcal{M}_\Delta)\leq 2g-1$, as discussed in Section \ref{section:hyperelliptic}.

It was shown in an addendum to \cite{nondegeneracy} that our claimed result holds if we replace $\dim(\mathbb{M}_{\Delta})$ with $\dim(\mathcal{M}_{\Delta})$.  Since any $\Delta$  with $\dim(\mathbb{M}_\Delta)=2g+1$ is  nonhyperelliptic and since only maximal polygons are under consideration, we may apply Corollary \ref{corollary:equality} to conclude our result holds, since $\dim(\mathbb{M}_{\Delta})=\dim(\mathcal{M}_{\Delta})$ for such polygons.
\end{proof}

\section{Nonmaximal nonhyperelliptic polygons}
\label{section:nonmaximal}

Throughout this section we let $\Delta$ be a nonmaximal nonhyperelliptic polygon.  As in the previous section, we will show that we may achieve a regular beehive triangulation of $\Delta$, this time using Proposition \ref{nonmaxfacerefinement}.  From here we will prove  Theorem \ref{Thm:MainTheorem} in the nonmaximal nonhyperelliptic case by arguing that a drop in rank in the $J$-matrix corresponds precisely to a drop in the dimension of the moduli space.

As in the maximal case, let $V(\Delta^{(1)})=\{v_1,\ldots v_n\}$ be ordered counterclockwise. Let the one-dimensional faces of $\Delta^{(1)}$ be $\{\tau_1,\ldots \tau_n\}$ where $\tau_i$ has endpoints $v_i,v_{i+1}$ (we will work with the indices mod $n$). Since $\Delta$ is nonmaximal we have $\Delta\subsetneq \Delta^{(0)}=\Delta^{(1)(-1)}$. The following lemma allows us to describe all faces on the boundary of $\Delta$.   Because we will be considering faces of both $\Delta$ and $\Delta^{(0)}$, for any one-dimensional face $\tau$ of $\Delta^{(-1)}$ we use $\tau^{(-1)}$ to refer to the relaxed face of $\Delta^{(0)}$ corresponding to $\tau$, and we let $\tilde{\tau}^{(-1)}$ to refer to $\tau^{(-1)}\cap \Delta$.  Recall by Lemma \ref{pushoutnonempty} that $\tilde{\tau}^{(-1)}$ is nonempty.

We can explicitly describe the faces in the boundary $\Delta$ as follows. If $\tilde{\tau_i}^{(-1)}\cap \tilde{\tau_{i+1}}^{(-1)}=\emptyset$, then there is an edge of lattice length one connecting them; let this edge be $\eta_i$. If ${\tau_i}^{(-1)}\cap \tilde{\tau_{i+1}}^{(-1)}\neq \emptyset$, by convention we will let $\eta_i=v_{i+1}^{(-1)}$.

To find a regular beehive triangulation of $\Delta$, we follow a similar strategy as in the maximal case: we will start with a regular subdivision,  and then further refine it. Let $\TTT_0$ be induced by the following height function $\omega_0:\Delta\cap\mathbb{Z}^2\rightarrow\mathbb{R}$:

\[\omega_0(p)= \begin{cases} 
      1 &  \textrm{if }p\in \partial\Delta \\
      0 & \textrm{otherwise.}
   \end{cases}
\]

The two-dimensional faces of $\TTT_0$ are all of the form $C_i:=\conv(\tau_i,\tilde{\tau_i}^{(-1)} )$ and $D_i:=\conv(v_{i+1},\eta_i )$, with one additional face corresponding to $\Delta^{(1)}$. Note that if $D_i$ is a two-dimensional face of $\TTT_0$, then it is a unimodular triangle.  See Figure \ref{figure:t0_illustrated} for an example.

\begin{figure}[hbt]
\centering
\includegraphics[scale=0.7]{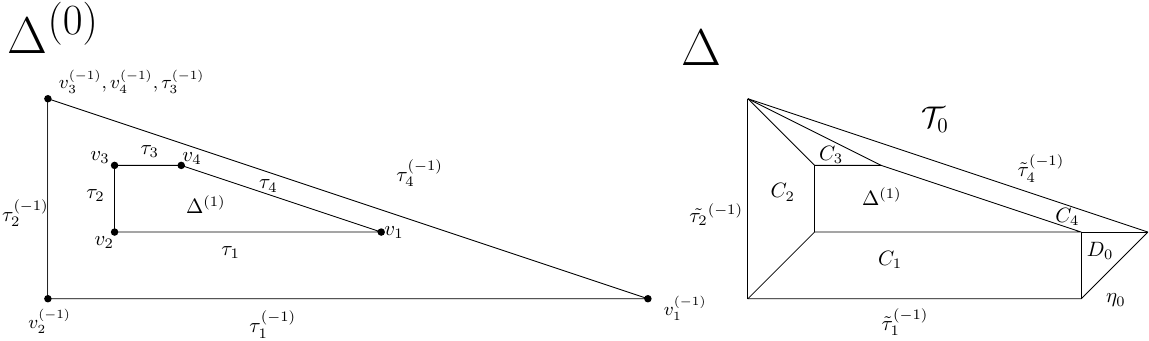}
\caption{A nonmaximal polygon $\Delta$ with the cells of $\TTT_0$ labeled on the right, with the corresponding maximal polygon $\Delta^{(0)}$ is on the left. The ordering of the vertices aligns with the chosen ordering in the proof of Proposition \ref{nonmaxfacerefinement}.}
\label{figure:t0_illustrated}
\end{figure}

\begin{prop}\label{nonmaxfacerefinement}
For any refinements $\mathcal{R}_1, \ldots \mathcal{R}_n$ of the faces $C_i$, there exists a regular refinement $\mathcal{R}$ of $\TTT_0$ such that $\mathcal{R}|_{C_i}=\mathcal{R}_i$ for every $i\in \{1, \ldots n \}$.
\end{prop}

We remark that this proposition is \emph{not} true for maximal polygons; a counterexample is illustrated in Figure \ref{figure:refinement_not_regular}.

\begin{figure}[hbt]
\centering
\includegraphics[scale=1]{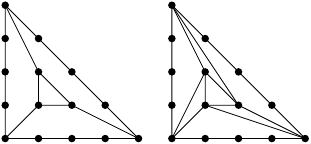}
\caption{A regular subdivision of a maximal polygon, and a choice of refinements of $C_1$, $C_2$, and $C_3$ that cannot appear in a regular triangulation}
\label{figure:refinement_not_regular}
\end{figure}

\begin{proof}
Similar to the maximal case, we will construct a height vector $\omega_i$ for every face $C_i$ such that $\mathcal{R}_i=\mathcal{R}(\Delta|_{C_i}, \omega_i)$, and at the intersection of two $C_i$'s the $\omega_i$'s agree. This creates a height vector $\omega$ satisfying the conditions of Lemma \ref{subdivision}. Choose $v_1\in V(\Delta^{(1)})$ such that $\eta_0$ has nonzero length; such an $\eta_0$ exists by Lemma \ref{pushoutnonempty} and the fact that $\Delta$ is nonmaximal. We then have that $\eta_0$ is in the boundary of $D_0$.

We  now note that any cell $C_i$ has lattice width $1$, and so by Lemma \ref{latticewidthone} we have that the refinement $\mathcal{R}_i$ is regular.  We will inductively choose our height vectors $\omega_i$, from $\omega_1$ to $\omega_n$.  Choose any height vector for $\omega_1$ inducing the subdivision $\mathcal{R}_1$ of $C_1$. Now assume that $\omega_i$ has been chosen for all $i\leq k$, where $k\leq n-1$. First consider the case where $k \neq n-1$. The two faces $C_k$ and $C_{k+1}$ intersect in either one or two lattice points (namely $v_{k+1}$, or  $v_{k+1}$ and $v_{k+1}^{(-1)}$), and so by Lemma \ref{threepointlemma}, $\omega_{k+1}$ can be chosen to agree with $\omega_k$.

Now consider the case where $k=n-1$. Note that we know by assumption that $C_n$ intersects with $C_1$ at only $v_1$ because of the existence of $D_0$. Thus $\omega_n$ need only agree with $\omega_{n-1}$ and $\omega_1$ on at most 3 points in total, and we can choose such a height vector for any subdivision $\mathcal{R}_n$ by Lemma \ref{threepointlemma}.  

Now define $\omega$ so that it agrees with each $\omega_i$ on each $C_i$, and is defined in any way on the other lattice points of $\Delta$.  By Lemma \ref{subdivision}, the refinement
\[\mathcal{R}=\bigcup_{C\in \mathcal{T}_0}\mathcal{R}(\Delta|_C,\omega|_C)\]
is regular, and satisfies $\mathcal{R}|_{C_i}=\mathcal{R}(\Delta|_{C_i},\omega|_{C_i})=\mathcal{R}_i$ for all $i$, as desired.
\end{proof}

We are now ready to prove that for $\Delta$ nonmaximal and nonhyperelliptic, we have $\dim(\mathbb{M}_\Delta)= \dim(\mathcal{M}_\Delta)$.

\begin{proof}[Proof of Theorem \ref{Thm:MainTheorem}, nonmaximal nonhyperelliptic case]

Let $\Delta^{(0)}=\Delta^{(1)(-1)}$. We already know that $\dim(\mathbb{M}_\Delta)$  is at most $\dim(\mathcal{M}_\Delta)$, and that $\dim(\mathbb{M}_{\Delta^{(0)}}) = \dim(\mathcal{M}_{\Delta^{(0)}})$ since $\Delta^{(0)}$ is maximal. 
If we can prove that $\dim(\mathbb{M}_{\Delta^{(0)}})-\dim(\mathbb{M}_{\Delta})\leq  |A|-\textrm{rank}(J)$, then we are done, since then we would have
\[\dim(\mathbb{M}_{\Delta})\geq \dim(\mathbb{M}_{\Delta^{(0)}})-|A|+\textrm{rank}(J)=\dim(\mathcal{M}_{\Delta^{(0)}})-|A|+\textrm{rank}(J)=\dim(\mathcal{M}_\Delta).\]


Consider the $J$ matrix of $\Delta$, whose generic entry in row $i$ and column $j$ is nonzero if and only if $a_i - c_j \in \Delta\cap\mathbb{Z}^2$. Let $\rank(J)=k$, and let $M$ be a $k\times k$ submatrix of $J$ with nonzero determinant. Relabelling our lattice points and our column vectors, we may assume that $M$ is the upper left $k\times k$ submatrix. Denoting the entry of $M$ at row $i$ and column $j$ as $m_{i,j}$, we have that
	\[ \det(M) = \sum_{\sigma\in S_k}\left(\text{sgn}(\sigma)\prod_{i=1}^km_{i,\sigma(i)}\right) \neq 0. \]
If every product $\prod_{i=1}^km_{i,\sigma(i)}$ were zero, then the determinant of $M$ would be zero, a contradiction. Therefore, there exists a permutation $\sigma_1$ such that $\prod_{i=1}^km_{i,\sigma_1(i)} \neq 0$ which implies $m_{i,\sigma_1(i)} \neq 0$ for every $i$ from 1 to $k$. Therefore, we know that $a_i - c_{\sigma_1(i)} \in \Delta\cap\mathbb{Z}^2$ for $i \in \{1,...,k\}$.  Thus  we have $k$ distinct column vectors of $\Delta^{(0)}$, each paired with a distinct lattice point of $A$.

Consider a regular beehive triangulation of $\Delta^{(0)}$, which maximizes $b_2+2b_3$.  We have that any vertex $v_i$ of $\Delta^{(1)}$ is connected to the pushout $v_i^{(-1)}$, as well as to the closest lattice points on $\tau_{i-1}$ and $\tau_{i}$ (if they exist).  Suppose a vertex $v_i^{(-1)}\in A$ is removed.  Either the dimension will drop by $1$ or by $0$, depending on whether a beehive triangulation for the smaller polygon can be found that is as optimal with respect to the edges $\tau_{i-1}$ and $\tau_i$.

We repeat this process, removing each point of $A$ one by one (always choosing a lattice point that is currently a vertex). The drop in dimension is therefore $|A|-N$, where $N$ number of times we were able to reconfigure our triangulation without losing dimension.  The number of ``free spaces'' we can use to fix our triangulation on an edge $\tau^{(-1)}_i$ is equal to $|\tau^{(-1)}_i|-1-|\tau_i|$. By Proposition \ref{proposition:counting_columns},  this is also equal to the number of column vectors associated to the face $\tau_i^{(-1)}$.  For each of the lattice points $a_1,\ldots,a_k$ that we delete, there is a distinct column vector contributing to the value $|\tau^{(-1)}_i|-1-|\tau_i|$ for some relevant $i$ which allows us to avoid a drop in dimension.  Thus, $N\geq k=\rank(J)$. This means that $\dim(\mathbb{M}_{\Delta^{(0)}})-\dim(\mathbb{M}_{\Delta})\leq  |A|-\textrm{rank}(J)$, completing the proof.

\end{proof}

\section{Hyperelliptic polygons}
\label{section:hyperelliptic}

Our main interest in this paper is for nonhyperelliptic polygons.  However,  we can quickly prove we do have $\dim(\mathbb{M}_\Delta)=\dim(\mathcal{M}_\Delta)$ for $\Delta$ maximal and hyperelliptic; we leave as work for future researchers to determine if this also holds for nonmaximal hyperelliptic polygons.  Our strategy is to construct a specific hyperelliptic polygon $H_g$ of genus $g$ that is contained in all maximal hyperelliptic polygons of genus $g$.  We will show in Proposition \ref{prop:hyp_dim} that the dimension of $\mathbb{M}_{H_g}$ is at least $2g-1$, giving us the same lower bound on  $\dim(\mathbb{M}_\Delta)$ for any maximal hyperelliptic polygon $\Delta$ of genus $g$.  This matches the dimension from the algebraic case, giving us the desired equality of dimensions.

In contrast the the nonhyperelliptic case, we have a very concrete classification of all maximal hyperelliptic polygons of fixed genus $g$.

\begin{lemma}
Let $g\geq 2$.  Up to lattice equivalence are exactly $g+2$ maximal hyperelliptic polygons of genus $g$, namely
\[E_k^{(g)}:=\textrm{conv}\left((0,0),(0,2),(g+k,0),(g+2-k,2)\right)\]
for $1\leq k\leq g+2$.
\end{lemma}
This was observed in \cite[\S 6]{BJMS}, and follows from picking the maximal polygons out from the classification of all hyperelliptic polygons in \cite{Koelman}; see also \cite[Theorem 10(c)]{pushingout}.  The genus $3$ polygons $E_1^{(3)}$ through $E_5^{(3)}$ are illustrated in Figure \ref{figure:hyperelliptic_polygons}.  It was shown in \cite{BJMS} that the hyperelliptic rectangle $E_1^{(g)}$ and the hyperelliptic triangle $E_{g+2}^{(g)}$ give rise to a family of graphs with dimension $2g-1$, matching the dimension of the moduli space $\mathcal{M}_g^{\textrm{hyp}}$ of all hyperelliptic algebraic curves of genus $g$; as argued there, it follows that $\dim(\mathbb{M}_\Delta)=\dim(\mathcal{M}_\Delta)$ for $\Delta\in\{E_1^{(g)},E_{g+2}^{(g)}\}$.  

	\begin{figure}[hbt]
   		 \centering
        \includegraphics[scale=0.7]{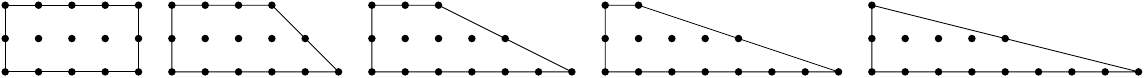}
		\caption{The maximal hyperelliptic polygons of genus 3}
		\label{figure:hyperelliptic_polygons}
	\end{figure}

In this section we wish to argue that the same holds for $\Delta=E_i^{(g)}$ for all $i$.  To find a lower bound on $\dim(\mathbb{M}_{P_i^{(g)}})$, we will consider the polygon 
\[H_g:=\textrm{conv}\left((0,0),(0,2),(g+1,0),(g+1,1)\right).\]
This (nonmaximal) hyperelliptic polygon is contained in $E_i^{(g)}$ for all $i$, so we have $\textrm{dim}(\mathbb{M}_{H_g})\leq \textrm{dim}(\mathbb{M}_{E_i^{(g)}})$.
 We now choose a particular unimodular triangulation $\mathcal{T}$ of $H_g$, guaranteed to be regular by \cite[Proposition 3.4]{regulartriangulations}.  For $0\leq j\leq g$, connect the point $(j,1)$ to $(j+1,1)$, splitting $H_g$ into an upper and lower half.  For the upper half, connect the point $(0,2)$ to all points of the form $(j,1)$.  For the lower half, connect the point $(j,0)$ to $(j,1)$ and $(j+1,i)$ for $0\leq j\leq g$.  The resulting unimodular triangulation $\mathcal{T}$ is illustrated for $g=3$ in Figure \ref{figure:hyperelliptic_triangulation}, along with a dual tropical curve.
 
 \begin{figure}[hbt]
   		 \centering
        \includegraphics[scale=1]{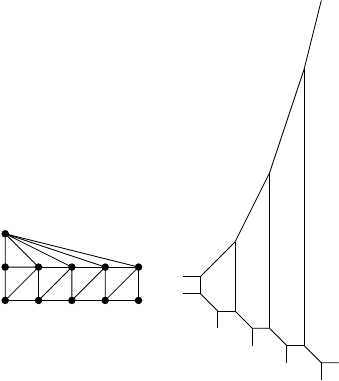}
		\caption{The triangulation $\mathcal{T}$, and a dual tropical curve}
		\label{figure:hyperelliptic_triangulation}
	\end{figure}
 
\begin{prop} \label{prop:hyp_dim} 

Letting $\mathcal{T}$ be the prescribed triangulation of $H_g$, we have $\dim(\mathbb{M}_\mathcal{T})=2g-1$.
\end{prop}
 
\begin{proof}
 We will prove this by explicitly finding the equalities and inequalities that define $\mathbb{M}_\mathcal{T}$.  Let $G$ denote the skeleton of a smooth tropical plane curve dual to $\mathcal{T}$; note that the combinatorial type of this skeleton is the bridgeless chain of loops, as discussed in \cite[\S 6]{BJMS}.  We label the edge lengths of such a graph as pictured in Figure \ref{figure:labeled_chain}:  the starting and ending loops have lengths $\ell_{s}$ and $\ell_e$, the common edges of bounded cycles have lengths $h_1,\ldots,h_{g-1}$, and the parallel edges of the $j^{th}$ cycle have upper length $u_j$ and lower length $w_j$ for $2\leq j\leq g-1$.  
 
 	\begin{figure}[hbt]
   		 \centering
        \includegraphics[scale=1]{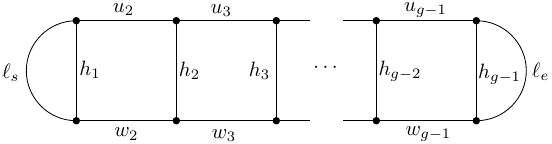}
		\caption{The length labels on the bridgeless chain of genus $g$}
		\label{figure:labeled_chain}
	\end{figure}
 
We claim that $\mathbb{M}_\mathcal{T}$ is defined by the usual nonnegativity requirements, along with the following equalities and inequalities, up to the natural symmetry of the graph:
\begin{enumerate}
    \item $u_j=w_j$ for all $j$;
    \item $h_1\leq\ell_s \leq 2h_1$ and $h_{g-1}\leq \ell_e$; and
    \item $h_j+ju_j\leq h_{j+1}\leq h_j+(j+1)u_j$.
\end{enumerate}
The fact that the equalities in (1) are necessary follows from \cite[Lemma 2.2]{thcitp}.  The inequalities in (2) amount to considering the choices on edge lengths for the first loop, interpolating from having most length in the vertical edge to most length in the edges with the horizontal components.  Finally, the length $h_{j+1}$ must be at least as large as $h_j$ plus whatever vertical translation is caused by the $u_j$ and $w_j$ edges; the $u_j$ edge will contribute exactly $jh_j$, and the $w_j$ edge can contribute (up to closure) anywhere between $0$ and $w_j$, depending on how much of the length goes into the horizontal edge and how much into the diagonal edge.  Given any set of lengths satisfying these bounds, we can build a tropical plane curve whose skeleton realizes these lengths by iteratively building one cycle after the next; it follows that these conditions are both necessary and sufficient.

The codimension of $\mathbb{M}_\mathcal{T}$ within $\mathbb{M}_g$ is equal to the number of linear equations, of which there are $g-2$.  Thus we have $\dim(\mathbb{M}_\mathcal{T})=(3g-3)-(g-2)=2g-1$, as claimed.
\end{proof}
 
 This allows us to prove the following corollary, which is the hyperelliptic case of Theorem \ref{Thm:MainTheorem}.
 
\begin{cor}
For any maximal hyperelliptic polygon $\Delta$ of genus $g\geq 2$, we have
\[\dim(\mathbb{M}_\Delta)=\dim(\mathcal{M}_\Delta)=2g-1.\]
\end{cor}

\begin{proof}
We have $\Delta=E_i^{(g)}$ for some $i$.  We know that $ \textrm{dim}(\mathbb{M}_{E_i^{(g)}})\geq\textrm{dim}(\mathbb{M}_{H_g})\geq 2g-1$ by the previous proposition.  Since $\mathcal{M}_{\Delta}\subset\mathcal{M}_g^{\textrm{hyp}}$, we have  $\dim(\mathcal{M}_{\Delta})\leq \dim(\mathcal{M}_g^{\textrm{hyp}})=2g-1$.  We thus have
\[2g-1\leq \dim(\mathbb{M}_\Delta)\leq \dim(\mathcal{M}_\Delta)\leq 2g-1,\]
so all inequalities must in fact be equalities, completing the proof.
\end{proof}

\bibliographystyle{amsplain-ac}
{\footnotesize{
\bibliography{paperbib}
}}
\end{document}